\documentclass[oneside,english]{amsart}
\usepackage[T1]{fontenc}
\usepackage[latin9]{inputenc}
\usepackage[a4paper]{geometry}
\usepackage{color}
\usepackage{xcolor}
\usepackage{amstext}
\usepackage{amsthm}
\usepackage{amssymb}
\usepackage{hyperref}
\hypersetup{
  colorlinks   = true,    
  urlcolor     = blue,    
  linkcolor    = teal,    
  citecolor    = purple      
}

\usepackage{comment}
\usepackage{setspace}
\makeatletter

\newcommand{\lyxmathsym}[1]{\ifmmode\begingroup\def\b@ld{bold}
  \text{\ifx\math@version\b@ld\bfseries\fi#1}\endgroup\else#1\fi}

\numberwithin{equation}{section}
\numberwithin{figure}{section}
\theoremstyle{plain}
\newtheorem{thm}{\protect\theoremname}

\makeatother

\usepackage{babel}
\providecommand{\theoremname}{Theorem}

\numberwithin{equation}{section}
\numberwithin{figure}{section}
\theoremstyle{plain}
\theoremstyle{definition}
\newtheorem{example}[thm]{\protect\examplename}
\theoremstyle{plain}
\newtheorem{prop}[thm]{\protect\propositionname}
\theoremstyle{remark}
\newtheorem{rem}[thm]{\protect\remarkname}

\makeatother

\usepackage{babel}
\providecommand{\examplename}{Example}
\providecommand{\propositionname}{Proposition}
\providecommand{\remarkname}{Remark}
\providecommand{\theoremname}{Theorem}

\begin{document}
\title{Higher Order Coercive Inequalities on Nilpotent Lie Groups}
 
\author{Esther Bou Dagher}
\address{Esther Bou Dagher\endgraf
CEREMADE, Universit{\'e} Paris Dauphine
}
\email{esther.bou-dagher@dauphine.psl.eu}
%
\author{Yifu Wang}
\address{Yifu Wang 	\endgraf 
School of Computing and Mathematical Sciences
\endgraf
University of Leicester}
\email{yw523@leicester.ac.uk }
%
\author{Boguslaw Zegarlinski}
\address{Boguslaw Zegarlinski\endgraf
          IMPAN 
		\endgraf
	 {\it E-mail address} { bzegarlinski@impan.pl}
	}
    \begin{abstract}
        In this work, we study regularity problems of certain Markov generators, which naturally appear in the context of analysis in functional spaces associated to probability measures on nilpotent Lie groups.
    \end{abstract}

\maketitle
 {\hypersetup{linkcolor=black}
\tableofcontents}

\section{Introduction} \label{sec1}
In this paper, in the context of 
$\mathbb{L}_p(\mu)$ and Orlicz spaces associated to a probability measure $\mu$ on nilpotent Lie groups, we discuss the problems of coercive inequalities of higher order as well as the relation of norms defined in terms of (sub)gradient and diffusion generators or higher order differential operators. In particular, this includes coercive inequalities of higher-order such as Poincar\'e-type and iterated Logarithmic Sobolev type inequalities,
which in Euclidean spaces were studied for example in \cite{WZ} and  \cite{Ro}, \cite{A}, respectively, see also references therein. We recall that it was shown already in \cite{Ro}, that given the following Logarithmic Sobolev inequality
\[
\int f^2\left|\log\frac{f^2}{\mu f^2}\right|^\beta  e^{-U}d\lambda
\leq  c\int |\nabla f|^2 e^{-U}d\lambda +
d \int |f|^2 e^{-U}d\lambda,
\]
with some constants $c,d,\beta\in(0,\infty)$ and a probability measure $d\mu\equiv e^{-U}d\lambda$ on $\mathbb{R}^n$,
using general arguments, one can show the following bound involving higher-order derivatives of the function 
\[
\int f^2\left|\log\frac{f^2}{\mu f^2}\right|^{m\beta}  e^{-U}d\lambda
\leq  C\sum_{|\alpha|\leq m} \int |\nabla f|^2 e^{-U}d\lambda +
D \int |f|^2 e^{-U}d\lambda,
\]
with some constants $C,D\in(0,\infty)$ in dependent of the function $f$ for which the right hand side is well defined. 
We discuss, with some detail, the derivation of the above fact in Section \ref{sec2} and in Section \ref{sec3} also including the case of nilpotent Lie groups,
where the gradient is replaced by the sub-gradient 
provided by the fields generating the corresponding nilpotent Lie algebra. 
The proof of Logarithmic Sobolev inequalities in the context of analysis on nilpotent Lie groups is much more intricate than in the classical Euclidean context; for some examples, we refer to
\cite{HZ}, \cite{BDZ},\cite{BDZhQZ}.\\
In \cite{A}, an elegant other scheme was developed which allowed for a nice derivation of iterated non-tight Logarithmic Sobolev type inequalities in the Euclidean framework. It was based on the following assumption on the interaction $U$ defining a probability measure
\[ \sum_{|\alpha|=2} |\nabla^\alpha f|\leq C(1+|\nabla U|)^{2-\delta},\]
with some constants $C\in(0,\infty)$ and $\delta\geq 0$. With this assumption, to which we refer later on as Adams' regularity condition,
the author got an inductive way of proving bounds of the following form
\[
\int |f|^p |\nabla U|^md\mu \leq C\sum_{|\alpha|\leq m} |\nabla^\alpha f|^p ,
\]
which played an important part in \cite{A} to prove non-tight coercive Logarithmic Sobolev type inequalities of higher order. (A tightened version of these inequalities was proven in \cite{WZ}.)\\
In Section \ref{sec4}, we discuss some other inductive bounds which do not involve the second-order condition. Part of the motivation for that is that in the area of analysis on nilpotent Lie groups with the sub-gradient replacing the full gradient,  Adams' regularity condition fails.
It is demonstrated in sections \ref{sec6} and \ref{sec7} with more examples (also in Euclidean spaces) provided in Section \ref{sec10}. \\
In Section \ref{sec5}, we discuss the iteration of the Poincar\'e inequality. This problem was tackled in the Euclidean space in \cite{WZ}, where it was shown that
the Poincar\'e inequality for a probability measure 
and full gradient in the Euclidean space allows, for a given function $f,$ a construction of an explicit polynomial $P_{f,m}$,  (called there "a statistical polynomial") such that
\[
\int |f-P_f|^pd\mu \leq C \int |\nabla^m f|^pd\mu,
\]
with $m\in\mathbb{N}$ and some constant $C\in(0,\infty)$. In Section \ref{sec5}, we provide a generalization of this fact to the case of nilpotent Lie groups, where the full gradient is replaced by the sub-gradient.\\

Regularity estimates are discussed in Section \ref{sec8}.
In particular, first, assuming the measure $\mu$ is quasi-regular in the sense of \cite{AH-K}, we consider a representation of the nilpotent Lie algebra in $\mathbb{L}_2(\mu)$ and give there a proof of global regularity in the sense of a bound on a Sobolev-type norm by a suitable higher norm involving the representation $\{V_j\equiv X_j-\frac12X_jU\}_{j\leq k}$ of the sub-gradient $\{X_j\}_{j\leq k}$ of the Lie algebra.
In classical PDE theory, working with the Lebesgue measure and H\"ormander fields, one obtains such results by localization to a compact set (\cite{OR},\cite{Kohn}) with the constants in such a bound depending very much on the size of the set (for example in \cite{Ha}). 
In our case, the bound is global.\\
To get back to the original fields, in Subsection \ref{sec8.1}, further analysis is necessary.
It is based on some improvements (of the notoriously defective $U$-bounds) in the case of nilpotent Lie groups. It also involves Hardy inequality for probability measures, discussed in Section \ref{sec9}.
We close with a proof of a nice relation between harmonic functions and the Kaplan norm in the case of the Heisenberg group.

\section{Higher Order $\beta-$Logarithmic Sobolev Inequality}\label{sec2}
Let $\mathbb{G}$ be a Carnot group with $n$ generators $X_{1},X_{2},...,X_{n},$
and subgradient $\nabla_{\mathbb{}}:=(X_{1},X_{2},...,X_{n}).$

We extend the Logarithmic Sobolev inequality to higher order with
a proof inspired by J. Rosen \cite{Ro}.
\begin{thm}
Given the following Logarithmic-Sobolev inequality:
\begin{equation}
\int|f|^{2}\left|\log\left(\frac{|f|^{2}}{\mu|f|^{2}}\right)\right|^{\beta}d\mu\leq C\mu|\nabla f|^{2},\label{eq:1}
\end{equation}

for $\beta\in(0,1).$  Then, for all $m\in\mathbb{N},$
\begin{equation}
\int|f|^{2}\left|\log\left(\frac{|f|^{2}}{\mu|f|^{2}}\right)\right|^{\beta m}d\mu\leq C\sum_{|\alpha|=0}^{m}\int|\nabla^{\alpha}f|^{2}d\mu,\label{eq:2}
\end{equation}

where $\nabla^{\alpha}f=(X_{1}^{\alpha_{1}}X_{2}^{\alpha_{2}}...X_{n}^{\alpha_{n}}f)$
such that $|\alpha|={\displaystyle \sum_{i=1}^{n}\alpha_{i}.}$
\end{thm}

\begin{proof}
The proof is by induction. We assume (\ref{eq:2}) is true for $m$,
so we need to show the statement is true for $m+1$: 
\begin{equation}
\int|f|^{2}\left|\log\left(\frac{|f|^{2}}{\mu|f|^{2}}\right)\right|^{\beta(m+1)}d\mu\leq C\sum_{|\alpha|=0}^{m+1}\int|\nabla^{\alpha}f|^{2}d\mu.\label{eq:3}
\end{equation}

By homogeneity, we can consider ${\displaystyle \int|f|^{2}d\mu=1}$
and prove 

\begin{equation}
\int|f|^{2}\left|\log|f|^{2}\right|^{\beta(m+1)}d\mu\leq C\left(1+\sum_{|\alpha|=0}^{m+1}\int|\nabla^{\alpha}f|^{2}d\mu\right).\label{eq:4}
\end{equation}

In fact, replace $f$ by ${\displaystyle \frac{f}{\left(\int|f|^{2}d\mu\right)^{\frac{1}{2}}}}$
in (\ref{eq:4}), to get

\[
\int\frac{|f|^{2}}{\left(\int|f|^{2}d\mu\right)}\left|\log\left(\frac{|f|^{2}}{\mu|f|^{2}}\right)\right|^{\beta(m+1)}d\mu\leq C\left(1+\sum_{|\alpha|=0}^{m+1}\int\left|\nabla^{\alpha}\left(\frac{f}{\left(\int|f|^{2}d\mu\right)^{\frac{1}{2}}}\right)\right|^{2}d\mu\right)
\]

\[
\frac{1}{\left(\int|f|^{2}d\mu\right)}\int|f|^{2}\left|\log\left(\frac{|f|^{2}}{\mu|f|^{2}}\right)\right|^{\beta(m+1)}d\mu\leq C\left(1+\frac{1}{\left(\int|f|^{2}d\mu\right)}\sum_{|\alpha|=0}^{m+1}\int\left|\nabla^{\alpha}f\right|^{2}d\mu\right),
\]

multiplying both sides by $\int|f|^{2}d\mu,$ we recover (\ref{eq:3}).
Hence, it suffices to prove (\ref{eq:4}). First consider $f^{2}\leq e,$
\[
\begin{array}{ll}
\int|f|^{2}\left|\log f^{2}\right|^{\beta(m+1)}d\mu & ={\displaystyle \int_{\{f^{2}<1\}}|f|^{2}\left|\log|f|^{2}\right|^{\beta(m+1)}d\mu+\int_{\{1\leq f^{2}\leq e\}}|f|^{2}\left|\log|f|^{2}\right|^{\beta(m+1)}d\mu}\\
 & {\displaystyle =\int_{\{f^{2}<1\}}|f|^{2}\left(\log\frac{1}{|f|^{2}}\right)^{\beta(m+1)}d\mu+\int_{\{1\leq f^{2}\leq e\}}|f|^{2}\left(\log|f|^{2}\right)^{\beta(m+1)}d\mu}\\
 & {\displaystyle {\displaystyle =\int_{\{f^{2}<1\}}|f|^{2}\left(\log\frac{1}{|f^{2}|}\right)^{\beta(m+1)}d\mu+\int_{\{1\leq f^{2}\leq e\}}|f|^{2}\left(\log|f^{2}|\right)^{\beta(m+1)}d\mu}}\\
 & \leq b\int_{\{f^{2}<1\}}1d\mu+\int_{\{1\leq f^{2}\leq e\}}ed\mu\\
 & \leq C.
\end{array},
\]
where we have used that for $|f|^{2}<1,$ ${\displaystyle \left(\log\frac{1}{|f^{2}|}\right)^{\beta(m+1)}\leq\frac{b}{|f|^{2}}}$
and $C=Z(b+e).$ Now we consider the case where $f^{2}\geq e:$
\[
\begin{array}{cl}
\int_{\{f^{2}\geq e\}}|f|^{2}\left|\log f^{2}\right|^{\beta(m+1)}d\mu & =\int_{\{f^{2}\geq e\}}f^{2}\left|\log(f^{2})\right|^{\beta m}\left|\log(f^{2})\right|^{\beta}d\mu\\
 & \leq\int_{\{f^{2}\geq e\}}f^{2}\left|\log(f^{2})\right|^{\beta m}\left|\log\left((f^{2})\left|\log(f^{2})\right|^{\beta m}\right)\right|^{\beta}d\mu,
\end{array}
\]
where the last step is true since $f^{2}\leq f^{2}\left|\log(f^{2})\right|^{\beta m}.$
Let $g=|f|\left|\log(f^{2})\right|^{\frac{\beta m}{2}},$ then 
\begin{equation}
\begin{array}{cl}
\int_{\{f^{2}\geq e\}}|f|^{2}\left|\log f^{2}\right|^{\beta(m+1)}d\mu & \leq \int_{\{f^{2}\geq e\}}g^{2}\left|\log\left(g^{2}\right)\right|^{\beta}d\mu\\
 & =\int_{\{f^{2}\geq e\}}g^{2}\left|\log\left({\displaystyle \frac{g^{2}}{\int g^{2}d\mu}}\right)+\log\left(\int g^{2}d\mu\right)\right|^{\beta}d\mu\\
 & \leq\int_{\{f^{2}\geq e\}}g^{2}\left|\log\left({\displaystyle \frac{g^{2}}{\int g^{2}d\mu}}\right)\right|^{\beta}d\mu+\int_{\{f^{2}\geq e\}}g^{2}\left|\log\left(\int g^{2}d\mu\right)\right|^{\beta}d\mu.
\end{array}\label{eq:5}
\end{equation}
Denote the first term on the right-hand side of (\ref{eq:5}) by ${\displaystyle A=}\int_{\{f^{2}\geq e\}}g^{2}\left|\log\left({\displaystyle \frac{g^{2}}{\int g^{2}d\mu}}\right)\right|^{\beta},$ and the second term by \textbf{${\displaystyle B=}\int_{\{f^{2}\geq e\}}g^{2}\left|\log\left(\int g^{2}d\mu\right)\right|^{\beta}d\mu.$ } We first replace $g$ in $B.$ By continuity since $\beta<1$, we
can find a positive number $\gamma$ such that ${\displaystyle 1<\gamma<\frac{1}{\beta}.}$
We get:
\[
\begin{array}{cl}
B & =\left(\int_{\{f^{2}\geq e\}}f^{2}\left|\log(f^{2})\right|^{\beta m}d\mu\right)\left|\log\left(\int f^{2}\left|\log(f^{2})\right|^{\beta m}d\mu\right)\right|^{\beta}\\
 & =\left(\int_{\{f^{2}\geq e\}}{\displaystyle f^{2}}{\displaystyle \left|\log(f^{2})\right|^{\frac{\beta m\gamma}{\gamma}}}d\mu\right)\left|\frac{\beta}{1-\beta\gamma}\log\left(\int{\displaystyle f^{2}}{\displaystyle \left|\log(f^{2})\right|^{\beta m}}d\mu\right)^{\frac{1-\beta\gamma}{\beta}}\right|^{\beta}\\
 & \leq{\displaystyle \left(\frac{\beta}{1-\beta\gamma}\right)^{\beta}}\left(\left(\int_{\{f^{2}\geq e\}}f^{2}d\mu\right)\int_{\{f^{2}\geq e\}}{\displaystyle \frac{f^{2}}{\left(\int_{\{f^{2}\geq e\}}f^{2}d\mu\right)}}{\displaystyle \left|\log(f^{2})\right|^{\frac{\beta m\gamma}{\gamma}}}d\mu\right)\left(\int f^{2}{\displaystyle \left|\log(f^{2})\right|^{\beta m}}d\mu\right)^{1-\beta\gamma}.
\end{array}
\]
Since $\beta\gamma<1,$ we have that $\phi(x)={\displaystyle x^{\beta\gamma}}$
is concave, so we apply Jensen's inequality:
\[
\begin{array}{cl}
B & \leq{\displaystyle \left(\frac{\beta}{1-\beta\gamma}\right)^{\beta}}\left(\int_{\{f^{2}\geq e\}}f^{2}d\mu\right)^{1-\beta\gamma}\left(\int_{\{f^{2}\geq e\}}{\displaystyle f^{2}}{\displaystyle \left|\log(f^{2})\right|^{\frac{m}{\gamma}}}d\mu\right)^{\beta\gamma}\left(\int{\displaystyle f^{2}}{\displaystyle \left|\log(f^{2})\right|^{\beta m}}d\mu\right)^{1-\beta\gamma}\\
 & {\displaystyle \leq\left(\frac{\beta}{1-\beta\gamma}\right)^{\beta}}\left(\int{\displaystyle f^{2}}{\displaystyle \left|\log(f^{2})\right|^{\frac{m}{\gamma}}}d\mu\right)^{\beta\gamma}\left(\int{\displaystyle f^{2}}{\displaystyle \left|\log(f^{2})\right|^{\beta m}}d\mu\right)^{1-\beta\gamma}
\end{array}
\]
Now, since ${\displaystyle \frac{1}{\gamma}<1}$ and $\beta<1,$ we use the inductive hypothesis (\ref{eq:2}), 
\begin{equation}
\begin{array}{cl}
B & \leq{\displaystyle \left(\frac{\beta}{1-\beta\gamma}\right)^{\beta}}\left(\int f^{2}{\displaystyle \left|\log(f^{2})\right|^{\frac{m}{\gamma}}}d\mu\right)^{\beta\gamma}\left(\int{\displaystyle f^{2}}{\displaystyle \left|\log(f^{2})\right|^{\beta m}}d\mu\right)^{1-\beta\gamma}\\
 & \leq{\displaystyle \left(\frac{\beta}{1-\beta\gamma}\right)^{\beta}}{\displaystyle \left(C_{1}+C_{1}\sum_{|\alpha|=0}^{m}\int|\nabla^{\alpha}f|^{2}d\mu\right)^{\beta\gamma}\left(C_{2}+C_{2}\sum_{|\alpha|=0}^{m}\int|\nabla^{\alpha}f|^{2}d\mu\right)^{1-\beta\gamma}}\\
 & \leq{\displaystyle \left(\frac{\beta}{1-\beta\gamma}\right)^{\beta}}{\displaystyle \left(C+C\sum_{|\alpha|=0}^{m}\int|\nabla^{\alpha}f|^{2}d\mu\right),}
\end{array}\label{eq:6}
\end{equation}
where $C=max\{C_{1},C_{2}\}.$ Now we go back to (\ref{eq:5}) to deal with term $A.$ Recalling
that $g=|f|\left|\log(f^{2})\right|^{\frac{\beta m}{2}}$ and using the base case (\ref{eq:1}),
\begin{equation}
\begin{array}{cl}
A & {\displaystyle =}\int_{\{f^{2}\geq e\}}g^{2}\left|\log\left({\displaystyle \frac{g^{2}}{\int g^{2}d\mu}}\right)\right|^{\beta}d\mu\\
 & =\int_{\{f^{2}\geq e\}}f^{2}\left|\log(f^{2})\right|^{\beta m}\left|\log\left({\displaystyle \frac{f^{2}\left|\log(f^{2})\right|^{\beta m}}{\int f^{2}\left|\log(f^{2})\right|^{\beta m}d\mu}}\right)\right|^{\beta}d\mu\\
 & \leq\int_{\{f^{2}\geq e\}}(f+e)^{2}\left|\log((f+e)^{2})\right|^{\beta m}\left|\log\left({\displaystyle \frac{(f+e)^{2}\left|\log((f+e)^{2})\right|^{\beta m}}{\int(f+e)^{2}\left|\log((f+e)^{2})\right|^{\beta m}d\mu}}\right)\right|^{\beta}d\mu\\
 & \leq C{\displaystyle \sum_{i=1}^{n}}{\displaystyle \int\left|X_{i}\left((f+e)\left|\log((f+e)^{2})\right|^{\frac{\beta m}{2}}\right)\right|^{2}d\mu}\\
 & =C{\displaystyle \sum_{i=1}^{n}\int\left|(X_{i}f)\left(\log((f+e)^{2})\right)^{\frac{\beta m}{2}}+\beta m\left(\log((f+e)^{2})\right)^{\frac{\beta m}{2}-1}X_{i}f\right|^{2}d\mu}\\
 & \leq2C(1+(\beta m)^{2}){\displaystyle \sum_{i=1}^{n}\int|X_{i}f|^{2}(\log((f+e)^{2})^{\beta m}d\mu}\\
 & =2C(1+(\beta m)^{2}){\displaystyle \sum_{i=1}^{n}\left(\int|X_{i}f|^{2}d\mu\right)\int\frac{|X_{i}f|^{2}}{\left(\int|X_{i}f|^{2}d\mu\right)}\left((\log(f^{2}+e))^{\beta m}\right)d\mu.}
\end{array}\label{eq:7}
\end{equation}
Using Young's inequality ${\displaystyle \int|U||V|d\mu\leq c\left(1+\int U(\log(U))^{\beta m}d\mu+\int e^{|V|^{\frac{1}{\beta m}}}d\mu\right)},$
with $U={\displaystyle \frac{|X_{i}f|^{2}}{\int|X_{i}f|^{2}d\mu}}$
and $V=\left|\log(f^{2}+e)\right|^{\beta m}$ for the term on the right-hand side of (\ref{eq:7}), we get: 
\begin{equation}
\begin{array}{cl}
A & \leq{\displaystyle 2Cc(1+(\beta m)^{2}){\displaystyle \sum_{i=1}^{n}\left(\int|X_{i}f|^{2}d\mu\right)\left(2+eZ+\int\frac{|X_{i}f|^{2}}{\left(\int|X_{i}f|^{2}d\mu\right)}\log\left(\frac{|X_{i}f|^{2}}{\left(\int|X_{i}f|^{2}d\mu\right)}\right)^{\beta m}d\mu)\right)}}\\
 & ={\displaystyle C_{1}{\displaystyle \sum_{i=1}^{n}\left(\int|X_{i}f|^{2}d\mu\right)+C_{2}{\displaystyle \sum_{i=1}^{n}}\int|X_{i}f|^{2}\log\left(\frac{|X_{i}f|^{2}}{\left(\int|X_{i}f|^{2}d\mu\right)}\right)^{\beta m}d\mu}}\\
 & \leq{\displaystyle C_{1}\sum_{i=1}^{n}\left(\int|X_{i}f|^{2}d\mu\right)+C_{2}C\sum_{|\alpha|=0}^{m}\int|\nabla^{\alpha}X_{i}f|^{2}d\mu}\\
 & \leq{\displaystyle \tilde{C}\sum_{|\alpha|=0}^{m+1}\int|\nabla^{\alpha}f|^{2}d\mu,}
\end{array}\label{eq:8}
\end{equation}
where the last inequality is true by the inductive hypothesis (\ref{eq:2}). Replacing (\ref{eq:6}) and (\ref{eq:8}) in (\ref{eq:5}), we get: 
\[
\int|f|^{2}\left|\log\left(\frac{|f|^{2}}{\mu|f|^{2}}\right)\right|^{\beta(m+1)}d\mu\leq C\sum_{|\alpha|=0}^{m+1}\int|\nabla^{\alpha}f|^{2}d\mu.
\]
\end{proof}

\section{
Higher Order $\beta-$Logarithmic Sobolev Inequality for $p>0$} \label{sec3} 
Let $\mathbb{G}$ be a Carnot group with $n$ generators $X_{1},X_{2},...,X_{n},$
and subgradient $\nabla_{\mathbb{}}:=(X_{1},X_{2},...,X_{n}).$ We extend the Logarithmic-Sobolev inequality to higher order with a proof inspired by J. Rosen \cite{Ro}.
\begin{thm}
Let $p>0.$ Given the following Logarithmic Sobolev inequality:
\begin{equation}
\int|f|^{p}\left|\log\left(\frac{|f|^{p}}{\mu|f|^{p}}\right)\right|^{\beta}d\mu\leq C\mu|\nabla f|^{p},\label{eq:1}
\end{equation}
for $\beta\in(0,1).$ Then, for all $m\in\mathbb{N},$
\begin{equation}
\int|f|^{p}\left|\log\left(\frac{|f|^{p}}{\mu|f|^{p}}\right)\right|^{\beta m}d\mu\leq C\sum_{|\alpha|=0}^{m}\int|\nabla^{\alpha}f|^{p}d\mu,\label{eq:2}
\end{equation}
where $\nabla^{\alpha}f=(X_{1}^{\alpha_{1}}X_{2}^{\alpha_{2}}...X_{n}^{\alpha_{n}}f)$
such that $|\alpha|={\displaystyle \sum_{i=1}^{n}\alpha_{i}.}$
\end{thm}

\begin{proof}
The proof is by induction. We assume (\ref{eq:2}) is true for $m$,
so we need to show the statement is true for $m+1$: 
\begin{equation}
\int|f|^{p}\left|\log\left(\frac{|f|^{p}}{\mu|f|^{p}}\right)\right|^{\beta(m+1)}d\mu\leq C\sum_{|\alpha|=0}^{m+1}\int|\nabla^{\alpha}f|^{p}d\mu.\label{eq:3}
\end{equation}
By homogeneity, we can consider ${\displaystyle \int|f|^{p}d\mu=1}$ and prove 
\begin{equation}
\int|f|^{p}\left|\log|f|^{p}\right|^{\beta(m+1)}d\mu\leq C\left(1+\sum_{|\alpha|=0}^{m+1}\int|\nabla^{\alpha}f|^{p}d\mu\right).\label{eq:4}
\end{equation}
In fact, replace $f$ by ${\displaystyle \frac{f}{\left(\int|f|^{p}d\mu\right)^{\frac{1}{p}}}}$
in (\ref{eq:4}), to get
\[
\int\frac{|f|^{p}}{\left(\int|f|^{p}d\mu\right)}\left|\log\left(\frac{|f|^{p}}{\mu|f|^{p}}\right)\right|^{\beta(m+1)}d\mu\leq C\left(1+\sum_{|\alpha|=0}^{m+1}\int\left|\nabla^{\alpha}\left(\frac{f}{\left(\int|f|^{p}d\mu\right)^{\frac{1}{p}}}\right)\right|^{p}d\mu\right)
\]
\[
\frac{1}{\left(\int|f|^{p}d\mu\right)}\int|f|^{p}\left|\log\left(\frac{|f|^{p}}{\mu|f|^{p}}\right)\right|^{\beta(m+1)}d\mu\leq C\left(1+\frac{1}{\left(\int|f|^{p}d\mu\right)}\sum_{|\alpha|=0}^{m+1}\int\left|\nabla^{\alpha}f\right|^{p}d\mu\right),
\]
multiplying both sides by $\int|f|^{p}d\mu,$ we recover (\ref{eq:3}).
Hence, it suffices to prove (\ref{eq:4}). First consider $|f|^{p}\leq e,$
\[
\begin{array}{ll}
\int|f|^{p}\left|\log|f|^{p}\right|^{\beta(m+1)}d\mu & ={\displaystyle \int_{\{|f|^{p}<1\}}|f|^{p}\left|\log|f|^{p}\right|^{\beta(m+1)}d\mu+\int_{\{1\leq|f|^{p}\leq e\}}|f|^{p}\left|\log|f|^{p}\right|^{\beta(m+1)}d\mu}\\
 & {\displaystyle =\int_{\{|f|^{p}<1\}}|f|^{p}\left(\log\frac{1}{|f|^{p}}\right)^{\beta(m+1)}d\mu+\int_{\{1\leq|f|^{p}\leq e\}}|f|^{p}\left(\log|f|^{p}\right)^{\beta(m+1)}d\mu}\\
 & {\displaystyle {\displaystyle =\int_{\{|f|^{p}<1\}}|f|^{p}\left(\log\frac{1}{|f|^{p}}\right)^{\beta(m+1)}d\mu+\int_{\{1\leq|f|^{p}\leq e\}}|f|^{p}\left(\log|f|^{p}\right)^{\beta(m+1)}d\mu}}\\
 & \leq b\int_{\{|f|^{p}<1\}}1d\mu+\int_{\{1\leq|f|^{p}\leq e\}}ed\mu\\
 & \leq C.
\end{array},
\]
where we have used that for $|f|^{p}<1,$ ${\displaystyle \left(\log\frac{1}{|f|^{p}}\right)^{\beta(m+1)}\leq\frac{b}{|f|^{p}}}$
and $C=Z(b+e).$ Now we consider the case where $|f|^{p}\geq e:$
\[
\begin{array}{cl}
\int_{\{|f|^{p}\geq e\}}|f|^{p}\left|\log|f|^{p}\right|^{\beta(m+1)}d\mu & =\int_{\{|f|^{p}\geq e\}}|f|^{p}\left|\log(|f|^{p})\right|^{\beta m}\left|\log(|f|^{p})\right|^{\beta}d\mu\\
 & \leq\int_{\{|f|^{p}\geq e\}}|f|^{p}\left|\log(|f|^{p})\right|^{\beta m}\left|\log\left(|f|^{p}\left|\log(|f|^{p})\right|^{\beta m}\right)\right|^{\beta}d\mu,
\end{array}
\]
where the last step is true since $|f|^{p}\leq|f|^{p}\left|\log(|f|^{p})\right|^{\beta m}.$
Let $g=|f|\left|\log(|f|^{p})\right|^{\frac{\beta m}{p}},$ then 
\begin{equation}
\begin{array}{cl}
\int_{\{|f|^{p}\geq e\}}|f|^{p}\left|\log|f|^{p}\right|^{\beta(m+1)}d\mu & \leq\int_{\{|f|^{p}\geq e\}}g^{p}\left|\log\left(g^{p}\right)\right|^{\beta}d\mu\\
 & =\int_{\{|f|^{p}\geq e\}}g^{p}\left|\log\left({\displaystyle \frac{g^{p}}{\int g^{p}d\mu}}\right)+\log\left(\int g^{p}d\mu\right)\right|^{\beta}d\mu\\
 & \leq\int_{\{|f|^{p}\geq e\}}g^{p}\left|\log\left({\displaystyle \frac{g^{p}}{\int g^{p}d\mu}}\right)\right|^{\beta}d\mu+\int_{\{|f|^{p}\geq e\}}g^{p}\left|\log\left(\int g^{p}d\mu\right)\right|^{\beta}d\mu.
\end{array}\label{eq:5}
\end{equation}
Denote the first term on the right-hand side of (\ref{eq:5}) by ${\displaystyle A=}\int_{\{|f|^{p}\geq e\}}g^{p}\left|\log\left({\displaystyle \frac{g^{p}}{\int g^{p}d\mu}}\right)\right|^{\beta},$ and the second term by \textbf{${\displaystyle B=}\int_{\{|f|^{p}\geq e\}}g^{p}\left|\log\left(\int g^{p}d\mu\right)\right|^{\beta}d\mu.$ } We first replace $g$ in $B.$ By continuity since $\beta<1$, we can find a positive number $\gamma$ such that ${\displaystyle 1<\gamma<\frac{1}{\beta}.}$ We get:
\[
\begin{array}{cl}
B & =\left(\int_{\{|f|^{p}\geq e\}}|f|^{p}\left|\log(|f|^{p})\right|^{\beta m}d\mu\right)\left|\log\left(\int|f|^{p}\left|\log(|f|^{p})\right|^{\beta m}d\mu\right)\right|^{\beta}\\
 & =\left(\int_{\{|f|^{p}\geq e\}}{\displaystyle |f|^{p}}{\displaystyle \left|\log(|f|^{p})\right|^{\frac{\beta m\gamma}{\gamma}}}d\mu\right)\left|\frac{\beta}{1-\beta\gamma}\log\left(\int|f|^{p}{\displaystyle \left|\log(|f|^{p})\right|^{\beta m}}d\mu\right)^{\frac{1-\beta\gamma}{\beta}}\right|^{\beta}\\
 & \leq{\displaystyle \left(\frac{\beta}{1-\beta\gamma}\right)^{\beta}}\left(\left(\int_{\{|f|^{p}\geq e\}}|f|^{p}d\mu\right)\int_{\{|f|^{p}\geq e\}}{\displaystyle \frac{|f|^{p}}{\left(\int_{\{|f|^{p}\geq e\}}|f|^{p}d\mu\right)}}{\displaystyle \left|\log(|f|^{p})\right|^{\frac{\beta m\gamma}{\gamma}}}d\mu\right)\left(\int|f|^{p}{\displaystyle \left|\log(|f|^{p})\right|^{\beta m}}d\mu\right)^{1-\beta\gamma}.
\end{array}
\]
Since $\beta\gamma<1,$ we have that $\phi(x)={\displaystyle x^{\beta\gamma}}$ is concave, so we apply Jensen's inequality:
\[
\begin{array}{cl}
B & \leq{\displaystyle \left(\frac{\beta}{1-\beta\gamma}\right)^{\beta}}\left(\int_{\{|f|^{p}\geq e\}}|f|^{p}d\mu\right)^{1-\beta\gamma}\left(\int_{\{|f|^{p}\geq e\}}|f|^{p}{\displaystyle \left|\log(|f|^{p})\right|^{\frac{m}{\gamma}}}d\mu\right)^{\beta\gamma}\left(\int|f|^{p}{\displaystyle \left|\log(|f|^{p})\right|^{\beta m}}d\mu\right)^{1-\beta\gamma}\\
 & {\displaystyle \leq\left(\frac{\beta}{1-\beta\gamma}\right)^{\beta}}\left(\int|f|^{p}{\displaystyle \left|\log(|f|^{p})\right|^{\frac{m}{\gamma}}}d\mu\right)^{\beta\gamma}\left(\int|f|^{p}{\displaystyle \left|\log(|f|^{p})\right|^{\beta m}}d\mu\right)^{1-\beta\gamma}
\end{array}
\]
Now, since ${\displaystyle \frac{1}{\gamma}<1}$ and $\beta<1,$ we use the inductive hypothesis (\ref{eq:2}), 
\begin{equation}
\begin{array}{cl}
B & \leq{\displaystyle \left(\frac{\beta}{1-\beta\gamma}\right)^{\beta}}\left(\int|f|^{p}{\displaystyle \left|\log(|f|^{p})\right|^{\frac{m}{\gamma}}}d\mu\right)^{\beta\gamma}\left(\int|f|^{p}{\displaystyle \left|\log(|f|^{p})\right|^{\beta m}}d\mu\right)^{1-\beta\gamma}\\
 & \leq{\displaystyle \left(\frac{\beta}{1-\beta\gamma}\right)^{\beta}}{\displaystyle \left(C_{1}+C_{1}\sum_{|\alpha|=0}^{m}\int|\nabla^{\alpha}f|^{p}d\mu\right)^{\beta\gamma}\left(C_{2}+C_{2}\sum_{|\alpha|=0}^{m}\int|\nabla^{\alpha}f|^{p}d\mu\right)^{1-\beta\gamma}}\\
 & \leq{\displaystyle \left(\frac{\beta}{1-\beta\gamma}\right)^{\beta}}{\displaystyle \left(C+C\sum_{|\alpha|=0}^{m}\int|\nabla^{\alpha}f|^{2}d\mu\right),}
\end{array}\label{eq:6}
\end{equation}
where $C=\max\{C_{1},C_{2}\}.$ Now we go back to to (\ref{eq:5}) to deal with term $A.$ Recalling
that $g=|f|\left|\log(|f|^{p})\right|^{\frac{\beta m}{p}}$ and using the base case (\ref{eq:1}),
\begin{equation}
\begin{array}{cl}
A & {\displaystyle =}\int_{\{|f|^{p}\geq e\}}g^{p}\left|\log\left({\displaystyle \frac{g^{p}}{\int g^{p}d\mu}}\right)\right|^{\beta}d\mu\\
 & =\int_{\{|f|^{p}\geq e\}}|f|^{p}\left|\log(|f|^{p})\right|^{\beta m}\left|\log\left({\displaystyle \frac{|f|^{p}\left|\log(|f|^{p})\right|^{\beta m}}{\int|f|^{p}\left|\log(|f|^{p})\right|^{\beta m}d\mu}}\right)\right|^{\beta}d\mu\\
 & \leq\int_{\{|f|^{p}\geq e\}}(|f|+e)^{p}\left|\log((|f|+e)^{p})\right|^{\beta m}\left|\log\left({\displaystyle \frac{(|f|+e)^{p}\left|\log((|f|+e)^{p})\right|^{\beta m}}{\int(|f|+e)^{p}\left|\log((|f|+e)^{p})\right|^{\beta m}d\mu}}\right)\right|^{\beta}d\mu\\
 & \leq C{\displaystyle \sum_{i=1}^{n}}{\displaystyle \int\left|X_{i}\left((|f|+e)\left|\log((|f|+e)^{p})\right|^{\frac{\beta m}{p}}\right)\right|^{p}d\mu}\\
 & =C{\displaystyle \sum_{i=1}^{n}\int\left|\frac{f}{|f|}(X_{i}f)\left(\log((|f|+e)^{p})\right)^{\frac{\beta m}{p}}+\beta m\frac{f}{|f|}\left(\log((|f|+e)^{p})\right)^{\frac{\beta m}{p}-1}X_{i}f\right|^{p}d\mu}\\
 & \leq\max\{1,2^{p-1}\}C(1+(\beta m)^{p}){\displaystyle \sum_{i=1}^{n}\int|X_{i}f|^{p}(\log((|f|+e)^{p})^{\beta m}d\mu}\\
 & =\max\{1,2^{p-1}\}C(1+(\beta m)^{p}){\displaystyle \sum_{i=1}^{n}\left(\int|X_{i}f|^{p}d\mu\right)\int\frac{|X_{i}f|^{p}}{\left(\int|X_{i}f|^{p}d\mu\right)}\left((\log(|f|+e)^{p})^{\beta m}\right)d\mu}.
\end{array}\label{eq:7}
\end{equation}
Using Young's inequality ${\displaystyle \int|U||V|d\mu\leq c\left(1+\int U(\log(U))^{\beta m}d\mu+\int e^{|V|^{\frac{1}{\beta m}}}d\mu\right)},$
with $U={\displaystyle \frac{|X_{i}f|^{p}}{\int|X_{i}f|^{p}d\mu}}$ and $V=\left|\log(|f|+e)^{p}\right|^{\beta m}$ for the term on the right-hand side of (\ref{eq:7}), we get: 

\begin{equation}
\begin{array}{cl}
A & \leq{\displaystyle \max\{1,2^{p-1}\}Cc(1+(\beta m)^{p}) \sum_{i=1}^{n}\left(\int|X_{i}f|^{p}d\mu\right)\cdot}\\
    & {\displaystyle\qquad\qquad\qquad\qquad\qquad\left(1+\max\{1,2^{p-1}\}+\max\{1,2^{p-1}\}e^{p}Z+\int\frac{|X_{i}f|^{p}}{\left(\int|X_{i}f|^{p}d\mu\right)}\log\left(\frac{|X_{i}f|^{p}}{\left(\int|X_{i}f|^{p}d\mu\right)}\right)^{\beta m}d\mu)\right)}\\
 & ={\displaystyle C_{1}{\displaystyle \sum_{i=1}^{n}\left(\int|X_{i}f|^{p}d\mu\right)+C_{2}{\displaystyle \sum_{i=1}^{n}}\int|X_{i}f|^{p}\log\left(\frac{|X_{i}f|^{p}}{\left(\int|X_{i}f|^{p}d\mu\right)}\right)^{\beta m}d\mu}}\\
 & \leq{\displaystyle C_{1}\sum_{i=1}^{n}\left(\int|X_{i}f|^{p}d\mu\right)+C_{2}C\sum_{|\alpha|=0}^{m}\int|\nabla^{\alpha}X_{i}f|^{p}d\mu}\\
 & \leq{\displaystyle \tilde{C}\sum_{|\alpha|=0}^{m+1}\int|\nabla^{\alpha}f|^{p}d\mu,}
\end{array}\label{eq:8}
\end{equation}

where the last inequality is true by the inductive hypothesis (\ref{eq:2}).
Replacing (\ref{eq:6}) and (\ref{eq:8}) in (\ref{eq:5}), we get: 

\[
\int|f|^{p}\left|\log\left(\frac{|f|^{p}}{\mu|f|^{p}}\right)\right|^{\beta(m+1)}d\mu\leq C\sum_{|\alpha|=0}^{m+1}\int|\nabla^{\alpha}f|^{p}d\mu.
\]
\end{proof}

\section{Inductive bounds} \label{sec4}

Let $d\mu\equiv e^{-U}d\lambda$. Suppose for some $a\in(0,\infty)$ and non-negative differentiable function $\mathcal{W}$, the following holds
\begin{equation}\label{a1}
 \mu \mathcal{W} f^2  \leq a\left(\mu|\nabla f|^2 +\mu f^2\right).
\end{equation}
 Then, for $n\in\mathbb{N}$, we have
 \[ \mu \mathcal{W} ^{n+1} f^2\leq a\mu\left( \left|\nabla\left(\mathcal{W} ^{\frac{n}{2}}f\right)\right|^2  +\left(\mathcal{W} ^{\frac{n}{2}}f\right)^2\right), \]
 whence,
 \[ \begin{split}
     \mu \mathcal{W}^{n+1} f^2&\leq 
     a\mu\left( 2 \mathcal{W} ^{n}\left|\nabla f\right|^2  + 2|\nabla\mathcal{W} ^{\frac{n}{2}}|^2 f^2 +\left(\mathcal{W} ^{\frac{n}{2}}f\right)^2 \right)\\
     &=a\mu\left( 2  \mathcal{W} ^{n}\left|\nabla f\right|^2  +  \frac{n^2}{2}  \mathcal{W} ^{n-2}|\nabla\mathcal{W} |^2 f^2 +\mathcal{W} ^n f^2\right).
 \end{split}\]
Suppose  
there exist constants $\varepsilon\in(0,\infty)$ and  $E_{\varepsilon}\in(0,\infty)$ such that, if $\mathcal{W}\neq 0$, then
\[ \frac{1}{\mathcal{W}}|\nabla \mathcal{W} |^2\leq \varepsilon \mathcal{W} ^2 +E_{\varepsilon}\tag{$\star$}\]
In this case, using 
($\star$), we have
\[ \begin{split}
     \mu \mathcal{W} ^{n+1} f^2&\leq 
     a\mu\left( 2  \mathcal{W} ^{n}\left|\nabla f\right|^2  +  \frac{n^2}{2}  \mathcal{W} ^{n-2}|\nabla\mathcal{W} |^2 f^2 +\mathcal{W} ^n f^2\right)\\
     &\leq 
     a\mu\left( 2  \mathcal{W} ^{n}\left|\nabla f\right|^2  +  \frac{\varepsilon n^2}{2}  \mathcal{W} ^{n+1}  f^2 + \mathcal{W} ^n f^2+ \frac{n^2}{2}E_\varepsilon  \mathcal{W}^{n-1}  f^2\right).
 \end{split}\]
 and hence, if $\varepsilon an^2/2<1$, we get
 \[ \begin{split}
     \mu \mathcal{W} ^{n+1} f^2& \leq 
    \frac{a}{1-\varepsilon an^2/2} a\mu\left( 2  \mathcal{W} ^{n}\left|\nabla f\right|^2    + \mathcal{W} ^n f^2+ \frac{n^2}{2}E_\varepsilon  \mathcal{W}^{n-1}  f^2\right).
 \end{split}\]
 Hence, by induction, we arrive at the following result:

 \begin{thm}
 Suppose  outside a  ball $B_R$,
there exists an $\varepsilon\equiv \varepsilon(R)\in(0,\infty)$ and a  constant $E_{\varepsilon}\in(0,\infty)$ such that
\[ \frac{1}{\mathcal{W} }|\nabla \mathcal{W} |^2\leq \varepsilon \mathcal{W} ^2 +E_{\varepsilon}.\]
If $\varepsilon a\frac{n^2}{2}<1$, 
then for every $m\leq n$
 there exists a constant $a_m\in(0,\infty)$ 
 \[\mu\left( \mathcal{W} ^{m} f^2\right) \leq a_m\mu\left( \sum_{k=1}^m \left|\nabla^k f\right|^2   +  f^2 \right).\]
 \end{thm}

\begin{example}
From \cite{BDZ}, we have
\[ \mu\left(\frac{g'(N)}{N^2} |f|^q \right)\leq C\mu\left( |\nabla f|^q \right) + D \mu\left( |f|^q \right).\]
Consider a differentiable function $F\equiv F(N)\geq 0$. Then we have
\[ \mu\left(\frac{g'(N)}{N^2}F^q(N) |f|^q \right)=  \mu\left(\frac{g'(N)}{N^2}|F(N)  f|^q \right)\leq 
C\mu\left( |\nabla (F(N)  f)|^q \right) + D \mu\left( F^q(N)  |f|^q \right).\]
\[\begin{split}
\leq C_1 \mu\left( |F(N)
|^q|\nabla   f |^q \right) +C_1 \mu\left(  |\nabla  ( F(N))|^q |f|^q\right) + D \mu\left( F^q(N)  |f|^q \right)\\
\leq C_1 \mu\left( |F(N)
|^q|\nabla   f |^q \right) + \mu\left(\left(   C_1|\nabla  ( F(N))|^q  + D   F^q(N) \right) |f|^q \right).
\end{split}
\]
If we assume that
\[ |F(N)|^q \leq \frac{g'(N)}{N^2} 
\]
and 
\[
C_1|\nabla  ( F(N))|^q  + D   F^q(N) \leq \frac{g'(N)}{N^2} ,
\]
then the iteration is finite and we get
\[ \mu\left(\frac{g'(N)}{N^2}F^q(N) |f|^q \right)
\leq C'' \mu\left(  |\nabla \nabla   f |^q \right) + C' \mu\left(  |  \nabla   f |^q \right)+ D' C'' \mu\left(  |  f |^q \right).
\]
\end{example}
For more examples, see also the end of Section \ref{sec9}.\\

\begin{example}
Let $d\mu\equiv e^{-U}d\lambda$ and
let
\[ \mathcal{V}\equiv \frac{1}{4}|\nabla U|^2-\frac12\Delta U\]
Suppose exist $C\in(0,\infty)$ such that
\[\mathcal{V_C}\equiv \mathcal{V} +C\geq 0 \]
and for some $a\in(0,\infty)$
\begin{equation}\label{a1}
 \mu f^2 \mathcal{V_C} \leq a\left(\mu|\nabla f|^2 +\mu f^2\right).
\end{equation}
 Then we have
 \[ \mu \mathcal{V_C}^{n+1} f^2\leq a\mu\left( \left|\nabla\left(\mathcal{V_C}^{\frac{n}{2}}f\right)\right|^2  +\left(\mathcal{V_C}^{\frac{n}{2}}f\right)^2\right) \]
 whence
 \[ \begin{split}
     \mu \mathcal{V_C}^{n+1} f^2&\leq 
     a\mu\left( 2 \mathcal{V_C}^{n}\left|\nabla f\right|^2  + 2|\nabla\mathcal{V_C}^{\frac{n}{2}}|^2 f^2 +\left(\mathcal{V_C}^{\frac{n}{2}}f\right)^2 \right)\\
     &=a\mu\left( 2  \mathcal{V_C}^{n}\left|\nabla f\right|^2  +  \frac{n^2}{2}  \mathcal{V_C}^{n-2}|\nabla\mathcal{V_C}|^2 f^2 +\mathcal{V_C}^n f^2\right).
 \end{split}\]
Special case $n+1=2$: 
\[\mu \left(\mathcal{V_C}^2 f^2\right)\leq 
a\mu\left( 2 \mathcal{V_C} \left|\nabla f\right|^2  + \frac{1}{\mathcal{V_C}}  \frac12 |\nabla\mathcal{V_C}|^2f^2 + \mathcal{V_C}f^2\right)
\]
Potential bound: Let $U\equiv U(N)$. Then we have
\[\begin{split}
|\nabla\mathcal{V_C}|^2 &= \sum_i \left|  
\sum_j \left(\frac12X_j U X_iX_j U -\frac12 X_i X_j^2 U \right)\right|^2\\
&=\sum_i \left|  
\sum_j \left(\frac12(X_j U) X_i(U'X_jN)  -\frac12 X_i (X_j  (U'X_jN) )\right)\right|^2\\
&=\sum_i \left|  
\sum_j \left(\frac12( U' X_jN) \left(U"(X_iN)X_jN+U'X_iX_jN\right)  -\frac12 X_i \left(  U"(X_jN)^2+U'X_j^2N  \right)\right)\right|^2\\
&_{= \sum_i \left|  
\sum_j \left(\frac12( U' X_jN) \left(U"(X_iN)X_jN+U'X_iX_jN\right)  -\frac12   \left(  U'''(X_iN) (X_jN)^2+2U" (X_jN)X_iX_jN +
U"(X_iN) X_j^2N +
U'X_iX_j^2N  \right)\right)\right|^2 }\\
&_{= \sum_i \left|  
\sum_j   \frac12 U'   \left(U"(X_iN)(X_jN)^2+U'(X_iX_jN ) (X_j N) \right)  
-\frac12   \left( U'''(X_iN) (X_jN)^2+2U" (X_jN)X_iX_jN +
U"(X_iN) X_j^2N +
U'X_iX_j^2N  \right)\right|^2 }\\
 \end{split}
\]
Based on homogeneity considerations, one should be able to claim that 
there exists a small constant $\varepsilon\in(0,\infty)$ and a  constant $E_{\varepsilon}\in(0,\infty)$ such that
\[ \frac{1}{\mathcal{V}_C}|\nabla \mathcal{V}_C|^2\leq \varepsilon \mathcal{V}_C^2 +E_{\varepsilon}\tag{C$\star$}\]
Then we get 
\[\mu \left(\mathcal{V_C}^2 f^2\right)\leq 
a\mu\left( 2 \mathcal{V_C} \left|\nabla f\right|^2  + \frac{\varepsilon}2 \mathcal{V_C}^2f^2  + \mathcal{V_C}f^2+\frac{E_{\varepsilon}}2f^2\right)
\]
and hence
\[\mu \left(\mathcal{V_C}^2 f^2\right)\leq 
\frac{a}{1-\frac{a\varepsilon}{2}}\mu\left( 2 \mathcal{V_C} \left|\nabla f\right|^2    + \mathcal{V_C}f^2+\frac{E_{\varepsilon}}2f^2\right).
\]
Thus using our inductive assumption \eqref{a1}, we arrive at
\[\mu \left(\mathcal{V_C}^2 f^2\right)\leq 
\max\left(3a,a+\frac{E_{\varepsilon}}{2}\right)\frac{2a}{2- a\varepsilon}  \mu\left(  \left|\nabla\nabla f\right|^2    +  |\nabla f|^2+ f^2\right).
\]
Now for the general case, using 
 (C$\star$), we have
\[ \begin{split}
     \mu \mathcal{V_C}^{n+1} f^2&\leq 
     a\mu\left( 2  \mathcal{V_C}^{n}\left|\nabla f\right|^2  +  \frac{n^2}{2}  \mathcal{V_C}^{n-2}|\nabla\mathcal{V_C}|^2 f^2 +\mathcal{V_C}^n f^2\right)\\
     &\leq 
     a\mu\left( 2  \mathcal{V_C}^{n}\left|\nabla f\right|^2  +  \frac{\varepsilon n^2}{2}  \mathcal{V_C}^{n+1}  f^2 + \mathcal{V_C}^n f^2+ \frac{n^2}{2}E_\varepsilon  \mathcal{V_C}^{n-1}  f^2\right).
 \end{split}\]
 and hence
 \[ \begin{split}
     \mu \mathcal{V_C}^{n+1} f^2& \leq 
    \frac{a}{1-\varepsilon an^2/2} a\mu\left( 2  \mathcal{V_C}^{n}\left|\nabla f\right|^2    + \mathcal{V_C}^n f^2+ \frac{n^2}{2}E_\varepsilon  \mathcal{V_C}^{n-1}   f^2\right).
 \end{split}\]
 \end{example}
 Hence, by induction, we arrive at the following result:
 \begin{thm}
 Suppose  outside a  ball $B_R$
there exists an $\varepsilon\equiv \varepsilon(R)\in(0,\infty)$ and a  constant $E_{\varepsilon}\in(0,\infty)$ such that
\[ \frac{1}{\mathcal{V}_C}|\nabla \mathcal{V}_C|^2\leq \varepsilon \mathcal{V}_C^2 +E_{\varepsilon}.\]
If $\varepsilon a\frac{n^2}{2}<1$, 
then
 there exists a constant $a_n\in(0,\infty)$ 
 \[\mu\left( \mathcal{V_C}^{n} f^2\right) \leq 
      a_n\mu\left( \sum_{k=1}^n \left|\nabla^k f\right|^2   +  f^2\right).\]
\end{thm}
 \textbf{Rockland Type Operators}\\
 Let $\nabla=(X_1,...,X_k)$ be subgradient defined by the generators of a nilpotent Lie algebra.
 The elementary $U$-bound in direction of a field $X_j$ reads
 \[
 \mu f^2\mathcal{V}_j\leq \mu|X_jf|^2
 \]
 where $\mathcal{V}_j\equiv \frac14|X_jU|^2-\frac12 X_j^2U$.
 If the following stronger bound holds
 \[
 \mu f^2|\mathcal{V}_j|\leq C\mu|X_jf|^2+\tilde D\mu f^2,
 \]
 then we can replace $|\mathcal{V}_j|$ by a strictly positive $\langle\mathcal{V}_j\rangle\equiv (|\mathcal{V}_j|^2+1)^\frac12$, and have a bound
 \[
 \mu f^2\langle\mathcal{V}_j\rangle\leq C\mu|X_jf|^2+D\mu f^2,
 \]
 with $D\equiv \tilde D+1$.
 This bound can now be iterated as follows. Firstly, this implies 
 \[\begin{split}
 \mu f^2\langle\mathcal{V}_j\rangle^2&\leq 2C\mu \left|X_j\left(f\langle\mathcal{V}_j\rangle^\frac12\right)\right|^2+D\mu  f^2  \langle\mathcal{V}_j\rangle \\
 &\leq 2C\mu(X_jf)^2\langle\mathcal{V}_j\rangle +2C\mu f^2\left(\frac1{2\langle\mathcal{V}_j\rangle}|\nabla \langle\mathcal{V}_j\rangle|^2\right)+D\mu f^2|\mathcal{V}_j|. 
 \end{split}
 \]
 Hence, if for some $\varepsilon\in(0,1)$  we have
 \[
  \frac1{2\langle\mathcal{V}_j\rangle}|\nabla \langle\mathcal{V}_j\rangle|^2 \leq \varepsilon \langle\mathcal{V}_j\rangle^2 +E,
 \]
then we obtain a bound
 \[
\mu f^2\langle\mathcal{V}_j\rangle^2\leq a \mu|X_j^2f|^2 + b  \mu|X_jf|^2 + c\mu f^2.
 \]
with some constants $a,b,c\in(0,\infty)$.
This implies that an iteration can be performed to get, for any $n\in\mathbb{N}$, with some constants $a_m\in(0,\infty)$, $m=0,..,n$, the following result 
 \[
 \mu f^2\mathcal{V}_j^{n} \leq \sum_{m=0}^n a_m \mu|X_j^{m}f|^2,  
 \]
 and consequently
 \[
 \mu \left(f^2\sum_{j\leq k}\mathcal{V}_j^{n} \right) \leq \sum_{m=0}^n a_m \sum_{j\leq k}\mu|X_j^{m}f|^2.  
 \]
Let us consider the following Rockland operator  of order $2n$ 
\[\textsc{R}\equiv (-1)^n\sum_{j=1,..,k} X_j^{2n}\]
A general theory of Logarithmic Sobolev type inequalities for Rockland type operators on nilpotent Lie group with a Haar measure $d\lambda$ was developed in \cite{CKR}.
Below, we discuss briefly the case of the probability measure
\[ d\mu \equiv e^{-U}d\lambda.
\]
First of all, according to \cite{RY} (and references therein), the following Sobolev inequality for the Rockland operator holds
\[\|f \|_{\mathbb{L}_q(\mathbb{G},\lambda)}\leq 
C \|\textsc{R}^\frac{a}{n} f \|_{\mathbb{L}_p(\mathbb{G},\lambda)} ,\qquad  1 < p < q < \infty, \; a = Q\left(
 \frac1{p} - \frac1{q}\right).
\]
In particular, for the Heisenberg group $\mathbb{H}_1$, with $Q=4$, we have
\[\|f \|_{\mathbb{L}_q(\mathbb{G},\lambda)}\leq 
C \|\textsc{R}^\frac{a}{n}f \|_{\mathbb{L}_2(\mathbb{G},\lambda)} ,\qquad  1 < p < q < \infty, \; a = Q\left(
 \frac1{p} - \frac1{q}\right).
\]
which for $p=2$ and $q,n=4$ reads
\[ \|f \|_{\mathbb{L}_4(\mathbb{G},\lambda)}\leq 
C \|\textsc{R}^\frac12 \|_{\mathbb{L}_2(\mathbb{G},\lambda)}.
\]
Hence, by standard arguments, see e.g. \cite{GZ}, we get
\[
\int f^2\log\frac{f^2}{\int f^2 d\lambda} d\lambda \leq
C \int |\textsc{R}^\frac{1}{n} f|^2 d\lambda + D \int |f|^2 d\lambda .
\]
Since for sufficiently smooth $f$
\[\int |\textsc{R}^\frac{1}{n} f|^2 d\lambda = \int  f\textsc{R}  f  d\lambda\leq \frac12\int  |\textsc{R}  f |^2 d\lambda +\frac12\int  | f |^2 d\lambda \]
we obtain
\[
\int f^2\log\frac{f^2}{\int f^2 d\lambda} d\lambda \leq
C' \int |\textsc{R} f|^2 d\lambda + D' \int |f|^2 d\lambda,
\]
with some $C',D'\in(0,\infty)$ independent of $f$ for which the right-hand side is well-defined.
Now, we can substitute $fe^{-\frac12U}$ in place of $f$ to get
\[
\int f^2\log\frac{f^2}{\int f^2 d\mu} d\mu
\leq
C" \int |\textsc{R} f|^2 d\mu +  \int f^2  U d\mu + A\sum_{j=1,..,k}\sum_{m=0,..,2n-1}\int |X_j^m f|^2 |X_j^{2n-m}U|^2 + D" \int |f|^2 d\mu,
\]
with some constants $C",D",A\in(0,\infty)$ independent of $f$.
This can be simplified using the quadratic form bounds developed above in this section. 
 In conclusion, we have the following result:
 \begin{thm}
 Suppose, for  $d\mu \equiv e^{-U}d\lambda$, the following quadratic form bounds are satisfied
 \[
 \mu\left( |X_j^m f|^2 |X_j^{2n-m}U|^2 \right)\leq 
 b_m \mu |\textsc{R} f|^2  +c_m \mu f^2,
 \]
 with some constants $b_m,c_m\in(0,\infty)$ independent of $f$,
then the following Logarithmic Sobolev inequality holds
\[
\mu\left( f^2\log\frac{f^2}{\mu f^2}  \right)
\leq
\tilde C\mu |\textsc{R} f|^2 + \tilde D\mu |f|^2, 
\]
with some constants $\tilde C,\tilde D\in(0,\infty)$.
\end{thm}
\section{Higher Order Poincar\'e Inequality} \label{sec5}

In this section, we generalize inductive proof of higher order Poincar\'e inequalities in the Euclidean setup considered in \cite{WZ} to the wealth of nilpotent Lie groups.\\
Let $\mathbb{G}$ be
a Lie group on $\mathbb{R}^{n}$ and let $\mathfrak{g}$ be its Lie
algebra. Then $\mathbb{G}$ is called a stratified
group, or Carnot group, if $\mathfrak{g}$ admits a vector space decomposition (stratification) of the form 

\[
\mathfrak{g}=\bigoplus_{j=1}^{r}V_{j},\mathrm{\;\;\;\;\;\;such\;that\;}\begin{cases}
[V_{1},V_{i-1}]=V_{i} & 2\leq i\leq r,\\{}
[V_{1},V_{r}]=\{0\}.
\end{cases}
\]

$\mathbb{R}^{n}$ can be split as $\mathbb{R}^{n}=\mathbb{R}^{n_{1}}\times...\times\mathbb{R}^{n_{r}},$
and the dilation $\delta_{\lambda}:\mathbb{R}^{n}\rightarrow\mathbb{R}^{n}$
\[
\delta_{\lambda}(x)=\delta_{\lambda}(x^{(1)},...,x^{(r)})=(\lambda x^{(1)},\lambda^{2}x^{(2)},...,\lambda^{r}x^{(r)}),\;\;\;\;\;\;\;\;\;x^{(i)}\in\mathbb{R}^{n_{i}},
\]
 is an automorphism of the group $\mathbb{G}$ for every $\lambda>0.$

We choose a Jacobian basis of $\mathfrak{g}$ adapted to the stratification
by selecting a basis of left-invariant vector fields $X_{1}^{(1)},...,X_{n_{1}}^{(1)}$
of $V_{1},$ $X_{1}^{(2)},...,X_{n_{2}}^{(2)}$ of $V_{2}$, ...,
and $X_{1}^{(r)},...,X_{n_{r}}^{(r)}$ of $V_{r},$ such that
at the the identity $e$ of $\mathbb{G}$, we have
\[
\begin{cases}
X_{i}^{(1)}(e)=e_{i} & \;\;1\leq i\leq n_{1},\\
X_{i}^{(2)}(e)=e_{n_{1}+i} & \;\;1\leq i\leq n_{2},\\
\;\;\;\vdots\\
X_{i}^{(r)}(e)=e_{n_{1}+n_{2}+...+n_{r-1}+i} & \;\;1\leq i\leq n_{r},
\end{cases}
\]
 where  $\{e_{h}\}_{h=1,...,n}$
is the standard basis of $\mathbb{R}^{n}$, and $n=n_{1}+n_{2}+...+n_{r}.$

In particular, $X_{1}^{(1)},...,X_{n_{1}}^{(1)}$ are the Jacobian
generators of $\mathbb{G}.$ $X_{j}^{(1)}(0)=\partial/\partial x_{j}|_{0}$
for $j=1,...,n_{1}$, and 
\[
\mathrm{rank}(\mathrm{Lie}{\{X_{1}^{(1)},...,X_{n_{1}}^{(1)}\}}(x))=n\;\;\;\;\;\;\forall x\in\mathbb{R}^{n}.
\]
Relatively to the basis, we write $z=(x^{(1)},...,x^{(r)})=(x_{1}^{(1)},x_{2}^{(1)}...,x_{n_{1}}^{(1)},...,x_{1}^{(r)},x_{2}^{(r)},...,x_{n_{r}}^{(r)}),$
or also $z\equiv(z_{h})_{1\leq h\leq n},$ where such coordinates
of $z\in\mathbb{G}$ comes from the exponential first kind representation 

\[
z=\exp\left(\sum_{j=1}^{r}\sum_{i=1}^{n_{j}}x_{i}^{(j)}X_{i}^{(j)}\right).
\]
Let
\[
(\zeta_{h})_{1\leq h\leq n}:=(\xi_{i_{1}}^{(1)},\xi_{i_{2}}^{(2)},...,\xi_{i_{r}}^{(r)})_{1\leq i_{1}\leq n_{1},1\leq i_{2}\leq n_{2},...,1\leq i_{r}\leq n_{r}}\in\mathfrak{g}^{*}
\]
 denote the dual basis of 
\[
(Z_{h})_{1\leq h\leq n}:=(X_{i_{1}}^{(1)},X_{i_{2}}^{(2)},...,X_{i_{r}}^{(r)})_{1\leq i_{1}\leq n_{1},1\leq i_{2}\leq n_{2},...,1\leq i_{r}\leq n_{r}}.
\]
Then, for any $1\leq h\leq n,$ we have that
\[
z_{h}=\zeta_{h}\circ\exp^{-1}(z):=\eta_{h}(z).
\]
Hence, we have that the general monomial can be written as
\[
z^{I}:=(x_{1}^{(1)})^{i_{1}^{(1)}}(x_{2}^{(1)})^{i_{2}^{(1)}}...(x_{n_{1}}^{(1)})^{i_{n_{1}}^{(1)}}...(x_{1}^{(r)})^{i_{1}^{(r)}}(x_{2}^{(r)})^{i_{2}^{(r)}}...(x_{n_{r}}^{(r)})^{i_{n_{r}^{(r)}}},
\]
where $|I|={\displaystyle \sum_{k=1}^{r}\sum_{l=1}^{n_{k}}i_{l}^{(k)}.}$

The vector valued operator $\nabla:=(X_{1},X_{2},...,X_{n_{1}})=(X_{1}^{(1)},X_{2}^{(1)},...,X_{n_{1}}^{(1)})$
is called the sub-gradient on $\mathbb{G},$ we see that
\[
X_{j}(\eta_{i})=\begin{cases}
1 & if\;\;\;i=j\\
0 & if\;\;\;i\neq j.
\end{cases}
\]

Define $\nabla^{\beta}=X_{1}^{\beta_{1}}X_{2}^{\beta_{2}}...X_{n_{1}}^{\beta_{n_{1}}}$
and $\eta^{\alpha}=\eta_{1}^{\alpha_{1}}\eta_{2}^{\alpha_{2}}...\eta_{n_{1}}^{\alpha_{n_{1}}},$
where $\alpha=(\alpha_{1},...,\alpha_{n_{1}})$ and $\beta=(\beta_{1},...,\beta_{n_{1}}).$
For $|\alpha|=|\beta|,$
\[
\nabla^{\beta}(\eta^{\alpha})=\begin{cases}
\beta!:=\beta_{1}!\beta_{2}!...\beta_{n_{1}}! & if\;\;\;\alpha=\beta\\
0 & if\;\;\;\alpha\neq\beta.
\end{cases}
\]

Choosing $w_{\beta}={\displaystyle \frac{\eta^{\beta}}{\beta!}},$
we obtain $\nabla^{\beta}(w_{\beta})=1.$

We have the Poincar\'e inequality $\mu|f-\mu f|^{2}\leq C\mu|\nabla f|^{2},$
and we want to show the higher order Poincar\'e inequality $\mu|f-\zeta|^{2}\leq C'\mu|\nabla^{k}f|^{2},$
for $k\in\mathbb{N}.$
\[
\begin{array}{ll}
\mu|\nabla^{k}f|^{2} & =\mu\left({\displaystyle \sum_{|\alpha|=k}|}\nabla^{\alpha}f|^{2}\right)\\
 & =\mu\left({\displaystyle \sum_{|\beta|=k-1}|}\nabla\nabla^{\beta}f|^{2}\right)\\
 & ={\displaystyle \sum_{|\beta|=k-1}}\mu\left({\displaystyle |}\nabla(\nabla^{\beta}f)|^{2}\right),
\end{array}
\]

using Poincar\'e inequality on $(\nabla^{\beta}f)$ and the fact
that $\nabla^{\beta}(w_{\beta})=1,$

\[
\begin{array}{ll}
 & \geq{\displaystyle \sum_{|\beta|=k-1}\frac{1}{C}}\mu\left|\nabla^{\beta}f-\mu(\nabla^{\beta}f)\right|^{2}\\
 & ={\displaystyle \sum_{|\beta|=k-1}\frac{1}{C}}\mu\left|\nabla^{\beta}\left(f-w_{\beta}\mu(\nabla^{\beta}f)\right)\right|^{2}\\
 & ={\displaystyle \sum_{|\beta|=k-1}\frac{1}{C}}\mu\left|\nabla^{\beta}\left(f-{\displaystyle \sum_{|\gamma|=k-1}}w_{\gamma}\mu(\nabla^{\gamma}f)\right)\right|^{2}\\
 & ={\displaystyle \sum_{|\beta_{2}|=k-2}\frac{1}{C}}\mu\left|\nabla\left(\nabla^{\beta_{2}}\left(f-{\displaystyle \sum_{|\gamma|=k-1}}w_{\gamma}\mu(\nabla^{\gamma}f)\right)\right)\right|^{2}
\end{array}
\]

using Poincar\'e inequality on $f_{1}:=\nabla^{\beta_{2}}\left(f-{\displaystyle \sum_{|\gamma|=k-1}}w_{\gamma}\mu(\nabla^{\gamma}f)\right)$
and the fact that $\nabla^{\beta_{2}}(w_{\beta_{2}})=1,$

\[
\begin{array}{cl}
 & \geq{\displaystyle \sum_{|\beta_{2}|=k-2}\frac{1}{C^{2}}}\mu\left|\nabla^{\beta_{2}}\left(f-{\displaystyle \sum_{|\gamma|=k-2}}w_{\gamma}\mu(\nabla^{\gamma}f)\right)\right|^{2}\\
 & \vdots\\
 & \geq{\displaystyle \frac{1}{C^{k}}\mu(f-\zeta_{k}),}
\end{array}
\]

where $\zeta_{k}$ is a polynomial depending on $f,k,$ and $\mu.$
\subsection{
Dual Polynomials for the Heisenberg group}

The lie algebra $\mathfrak{g}$ of the Heisenberg group $\mathbb{H}^{n}$
is spanned by the following left-invariant vector fields:
\[
X_{i}=\partial_{x_{i}}+\frac{(-1)^{i}}{2}x_{2n-i+1}\partial_{t},\;\;\;for\;\;i=1,...,2n,\;\;\;and\;\;T=\partial_{t}.
\]
The dual space of $\mathfrak{g}$ is denoted by $\Lambda^{1}\mathfrak{g}.$
The basis of $\Lambda^{1}\mathfrak{g},$ dual to the basis $\{W_{1},...,W_{2n+1}\}:=\{X_{1},...,X_{2n,}T\},$
is the family of covectors 
\[
\{\zeta_{1},...,\zeta_{2n+1}\}:=\bigg\{ dx_{1},...,dx_{2n},dt-\frac{1}{2}\sum_{i=1}^{n}\bigg(x_{i}dx_{2n-i+1}+(-1)^{i}x_{2n-i+1}dx_{i}\bigg)\bigg\}.
\]
We denote by $\langle.,.\rangle$ the inner product in $\Lambda^{1}\mathfrak{g}$
that makes $\{\zeta_{1},...,\zeta_{2n+1}\}$ an orthonormal basis.
We note that, for $1\leq h,i\leq2n+1,$
\begin{equation}
\langle\zeta_{h},W_{i}\rangle=\begin{cases}
1 & if\;\;h=i\\
0 & if\;\;h\neq i,
\end{cases}\label{eq:10}
\end{equation}
(see Section 1.2 of \cite{key-4}). Let $z:=(x_{1},...,x_{2n},t)\in\mathbb{G}.$
Using the exponential of first-kind representation, 

\[
z=\exp\left(\sum_{i=1}^{2n}x_{i}X_{i}+tT\right).
\]
For any $1\leq h\leq2n+1,$ we define
\begin{equation}
\eta_{h}(z):=\zeta_{h}\circ\exp{}^{-1}(z).\label{eq:11}
\end{equation}
We can write (\ref{eq:11}) explicitly using (\ref{eq:10}), so 
\begin{equation}
\eta_{h}(z)=\begin{cases}
x_{h} & if\;\;1\leq h\leq2n\\
t & if\;\;h=2n+1.
\end{cases}\label{eq:12}
\end{equation}
For the Poincar\'e downhill induction, we are interested in
\begin{equation}
\eta^{\alpha}:=\eta_{1}^{\alpha_{1}}\eta_{2}^{\alpha_{2}}...\eta_{2n}^{\alpha_{2n}},\label{eq:13}
\end{equation}
where $\alpha=(\alpha_{1},...,\alpha_{2n})$ and $|\alpha|={\displaystyle \sum_{i=1}^{2n}\alpha_{i}.}$
Using (\ref{eq:12}), we can explicitly write (\ref{eq:13}): 
\begin{equation}
\eta^{\alpha}(z)=\eta_{1}^{\alpha_{1}}(z)\eta_{2}^{\alpha_{2}}(z)...\eta_{2n}^{\alpha_{2n}}(z)=x_{1}^{\alpha_{1}}x_{2}^{\alpha_{2}}...x_{2n}^{\alpha_{2n}}.\label{eq:14}
\end{equation}
Let $\nabla^{\beta}:=X_{1}^{\beta_{1}}X_{2}^{\beta_{2}}...X_{2n}^{\beta_{2n}},$
where $\beta=(\beta_{1},...,\beta_{2n}).$ For $|\alpha|=|\beta|,$
by (\ref{eq:14}),
\[
\nabla^{\beta}(\eta^{\alpha})=\nabla^{\beta}(x_{1}^{\alpha_{1}}x_{2}^{\alpha_{2}}...x_{2n}^{\alpha_{2n}})=\begin{cases}
\beta!:=\beta_{1}!\beta_{2}!...\beta_{n_{1}}! & if\;\;\;\alpha=\beta\\
0 & if\;\;\;\alpha\neq\beta.
\end{cases}
\]

\section{Adams' Regularity Condition in the Heisenberg Group}\label{sec6}

In this section we consider the general Heisenberg group $\mathbb{H}^{n}$, that
is a group isomorphic to $\mathbb{R}^{2n+1},$ with the group operation
is defined by
\[
(x,z)\circ(x',z')=\left(x_{i}+x'_{i},z+z'+\frac{1}{2}\langle\Lambda x,x'\rangle\right)_{i=1,...,2n},
\]
 where $x=(x_{1},x_{2},...,x_{2n})\in\mathbb{R}^{2n},$ $x'=(x'_{1},x'_{2},...,x'_{2n})\in\mathbb{R}^{2n},$
and $z,z'\in\mathbb{R},$ and
\[
\Lambda=\mathrm{diag}\left\{ \left(\begin{array}{cc}
0 & -1\\
1 & 0
\end{array}\right),...,\left(\begin{array}{cc}
0 & -1\\
1 & 0
\end{array}\right)\right\} .
\]
The Lie algebra $\mathfrak{g}$ of the Heisenberg group $\mathbb{H}^{n}$ is spanned
by the following left-invariant vector fields: 
\[
X_{i}=\frac{\partial}{\partial x_{i}}+\frac{(-1)^{i}}{2}x_{2n-i+1}\frac{\partial}{\partial z},\;\;\;for\;\;i=1,...,2n,\;\;\;and\;\;Z=\frac{\partial}{\partial z}.
\]
 The sub-gradient is defined as $\nabla_{\mathbb{H}^{n}}=(X_{1},X_{2},...,X_{2n}),$
and the sub-Laplacian is defined as $\Delta={\sum_{i=1}^{2n}X_{i}^{2}.}$ 

We consider examples of $U$ defined as a function of a
(smooth) distance from the neutral element provided by the following
Kaplan Norm 
\[
{N=\left(|x|^{4}+16z^{2}\right)^{\frac{1}{4}}.}
\]
 where 
\[
|x|^{2}\equiv\sum_{i=1}^{2n}x_{i}^{2}
\]
 Then we have the following relations 
\[
X_{i}N=\frac{|x|^{2}x_{i}+(-1)^{i}x_{2n-i+1}4z}{N^{3}}
\]

if $j\neq2n-i+1$

\[
\begin{split}X_{j}X_{i}N=-\frac{3\left(|x|^{2}x_{i}+(-1)^{i}x_{2n-i+1}4z\right)\left(|x|^{2}x_{j}+(-1)^{j}x_{2n-j+1}4z\right)}{N^{7}}\\
+\frac{2x_{j}x_{i}+2(-1)^{i+j}x_{2n-j+1}x_{2n-i+1}}{N^{3}},
\end{split}
\]

and 
\[
X_{2n-i+1}X_{i}N=-\frac{3\left(|x|^{2}x_{i}+(-1)^{i}x_{2n-i+1}4z\right)\left(|x|^{2}x_{2n-i+1}-(-1)^{i}x_{i}4z\right)}{N^{7}}+\frac{4(-1)^{i}z}{N^{3}},
\]

\[
X_{i}X_{2n-i+1}N=-\frac{3\left(|x|^{2}x_{i}+(-1)^{i}x_{2n-i+1}4z\right)\left(|x|^{2}x_{2n-i+1}-(-1)^{i}x_{i}4z\right)}{N^{7}}-\frac{4(-1)^{i}z}{N^{3}}.
\]

Adams' regularity condition states that $\exists$ $\epsilon,C\in(0,\infty)$
such that
\begin{equation}
\sum_{|\alpha|=2}|\nabla^{\alpha}U|\leq C(1+|\nabla U|)^{2-\epsilon}.\label{eq:a}
\end{equation}

The left hand side of (\ref{eq:a}) is given by 
\[
\begin{split}\sum_{|\alpha|=2}|\nabla^{\alpha}U| & =\sum_{|\alpha|=2}|U'\nabla^{\alpha}N+U"(\nabla^{\alpha_{1}}N\cdot\nabla^{\alpha_{2}}N)|\\
 & \leq|U'(N)|\left(\sum_{i=1}^{2n}|X_{2n-i+1}X_{i}N|+\sum_{i=1}^{2n}|X_{i}X_{2n-i+1}N|+2\sum_{i=1}^{2n}\sum_{j=1_{,}j\neq2n-i+1}^{2n}|X_{j}X_{i}N|\right)\\
 & +|U"||\nabla N|^{2}
\end{split}
\]
 Since $|x|^{2}\leq N^{2},$ $|z|\leq N^{2},$ and $|x_{i}|\leq N\;\;\forall\;\;1\leq i\leq2n,$
then 
\begin{equation}
\sum_{|\alpha|=2}|\nabla^{\alpha}N|\leq\frac{\tilde{C}}{N}.\label{eq:b}
\end{equation}
 On the other hand, the right-hand side of (\ref{eq:a}) is 
\[
C(1+|\nabla U|)^{2-\epsilon}=C\left(1+|U'(N)|\left(\sum_{i=1}^{2n}\frac{\left|x|^{2}x_{i}+(-1)^{i}x_{2n-i+1}4z\right)^{2}}{N^{6}}\right)^{\frac{1}{2}}\right)^{2-\epsilon}
\]
 
\[
=C\left(1+|U'(N)|\left(\sum_{i=1}^{2n}\frac{|x|^{4}x_{i}^{2}+16x_{2n-i+1}^{2}z^{2}}{N^{6}}\right)^{\frac{1}{2}}\right)^{2-\epsilon}
\]
 
\[
=C\left(1+|U'(N)|\frac{|x|}{N}\right)^{2-\epsilon}
\]
 So, using (\ref{eq:b}), outside the unit ball $\{N<1\},$ we see
that Adams' regularity condition (\ref{eq:1}) cannot be satisfied in a thin region along the $z-$ axis.
Summarizing, we have the following result:
\begin{thm}
Let $\mathbb{H}^n$ be Heisenberg group and let 
${N=\left(|x|^{4}+16z^{2}\right)^{\frac{1}{4}}}$
be the corresponding Kaplan norm.
Then, the interaction $U\equiv U(N)$ given as smooth function of the Kaplan norm $N$ does not satisfy
Adams' regularity condition.\\
\end{thm}
\section{
Adams' Regularity Condition for Dual Polynomials}
\label{sec7}
In this section, we consider interactions given in terms of dual polynomials and show that they fail Adams' regularity condition.\\
Let $\mathbb{G}$ be a Carnot group with $n$ generators $X_{1},X_{2},...,X_{n},$
subgradient $\nabla_{\mathbb{}}:=(X_{1},X_{2},...,X_{n}),$
and $\nabla^{\beta}=X_{1}^{\beta_{1}}X_{2}^{\beta_{2}}...X_{n_{1}}^{\beta_{n_{1}}}$
for $\beta=(\beta_{1},...,\beta_{n_{1}}).$ We define the dual polynomials
$\eta^{\alpha}=\eta_{1}^{\alpha_{1}}\eta_{2}^{\alpha_{2}}...\eta_{n_{1}}^{\alpha_{n_{1}}}$
with $\alpha=(\alpha_{1},...,\alpha_{n_{1}})$ as in Section \ref{sec3}, and
we have
\begin{equation}
X_{j}(\eta_{i})=\begin{cases}
1 & if\;\;\;i=j\\
0 & if\;\;\;i\neq j.
\end{cases}\label{eq:c1}
\end{equation}

For $|\alpha|=|\beta|,$ 
\[
\nabla^{\beta}(\eta^{\alpha})=\begin{cases}
\beta!:=\beta_{1}!\beta_{2}!...\beta_{n_{1}}! & if\;\;\;\alpha=\beta\\
0 & if\;\;\;\alpha\neq\beta.
\end{cases}
\]

Let $U:[0,\infty)\rightarrow [0,\infty)$ be a twice-differentiable and increasing function. Adams' regularity condition requires that $\exists$ $\epsilon,C\in(0,\infty)$
such that 
\begin{equation}
\sum_{|\alpha|=2}|\nabla^{\alpha}U|\leq C(1+|\nabla U|)^{2-\epsilon}.\label{eq:a}
\end{equation}

By (\ref{eq:c1}), we have that 
\begin{equation}\begin{split}
X_{j}U(\eta^{\alpha})&=U'(\eta^{\alpha})X_{j}(\eta^{\alpha})\\
&=U'(\eta^{\alpha})X_{j}(\eta_{1}^{\alpha_{1}}\eta_{2}^{\alpha_{2}}...\eta_{n_{1}}^{\alpha_{n_{1}}})
\\
&=\alpha_{j}U'(\eta^{\alpha})\eta_{1}^{\alpha_{1}}\eta_{2}^{\alpha_{2}}...\eta_{j}^{\alpha_{j}-1}..\eta_{n_{1}}^{\alpha_{n_{1}}}.\label{eq:r1}
\end{split}
\end{equation}

Hence, 
\[
|\nabla U|=|U'(\eta^{\alpha})|\bigg(\sum_{j=1}^{n_{1}}\left(\alpha_{j}\frac{\eta^{\alpha}}{\eta_{j}}\right)^{2}\bigg)^{\frac{1}{2}}.
\]

Also, using (\ref{eq:r1}),
\[
X_{i}X_{j}U(\eta^{\alpha})=X_{i}\bigg(\alpha_{j}U'(\eta^{\alpha})\eta_{1}^{\alpha_{1}}\eta_{2}^{\alpha_{2}}...\eta_{j}^{\alpha_{j}-1}..\eta_{n_{1}}^{\alpha_{n_{1}}}\bigg).
\]

Using (\ref{eq:c1}),
\[
X_{i}X_{j}U(\eta^{\alpha})=\begin{cases}
\alpha_{i}\alpha_{j}\frac{\eta^{\alpha}}{\eta_{i}\eta_{j}}U'(\eta^{\alpha})+\alpha_{i}\alpha_{j}\frac{(\eta^{\alpha})^{2}}{\eta_{i}\eta_{j}}U''(\eta^{\alpha}) & if\;\;i\neq j\\
\alpha_{j}(\alpha_{j}-1)\frac{\eta^{\alpha}}{\eta_{j}^{2}}U'(\eta^{\alpha})+\alpha_{j}^{2}\frac{(\eta^{\alpha})^{2}}{\eta_{j}^{2}}U''(\eta^{\alpha}) & if\;\;i=j
\end{cases}.
\]

Therefore,
\[
\sum_{|\alpha|=2}|\nabla^{\alpha}U|=\sum_{j=1}^{n_{1}}\bigg|\alpha_{j}(\alpha_{j}-1)\frac{\eta^{\alpha}}{\eta_{j}^{2}}U'(\eta^{\alpha})+\alpha_{j}^{2}\frac{(\eta^{\alpha})^{2}}{\eta_{j}^{2}}U''(\eta^{\alpha})\bigg|+\sum_{i=1}^{n_{1}}\sum_{j=1,j\neq i}^{n_{1}}\bigg|\alpha_{i}\alpha_{j}\frac{\eta^{\alpha}}{\eta_{i}\eta_{j}}U'(\eta^{\alpha})+\alpha_{i}\alpha_{j}\frac{(\eta^{\alpha})^{2}}{\eta_{i}\eta_{j}}U''(\eta^{\alpha})\bigg|
\]
\[
\leq\sum_{j=1}^{n_{1}}|U'(\eta^{\alpha})|\bigg|\alpha_{j}(\alpha_{j}-1)\frac{\eta^{\alpha}}{\eta_{j}^{2}}\bigg|+\sum_{i=1}^{n_{1}}\sum_{j=1,j\neq i}^{n_{1}}|U'(\eta^{\alpha})|\bigg|\alpha_{i}\alpha_{j}\frac{\eta^{\alpha}}{\eta_{i}\eta_{j}}\bigg|+\sum_{i=1}^{n_{1}}\sum_{j=1}^{n_{1}}|U''(\eta^{\alpha})|\bigg|\alpha_{i}\alpha_{j}\frac{(\eta^{\alpha})^{2}}{\eta_{i}\eta_{j}}\bigg|.
\]
Hence, Adams' regularity condition is not satisfied. 
%
%
%
\section{Towards Global Regularity} \label{sec8}
Let $d\mu\equiv e^{-U}d\lambda,$ and
define the fields
\[
V_j\equiv V(X_j,U)\equiv X_j -\frac12 X_jU,
\]
with a first-order differential operator $X_j$.
We have that
\[\tag{*}
\mu (fV_jg)= - \mu ((V_jf)g).
\]
We also have that
\[\begin{split}
    [V_j,V_k]&= [X_j -\frac12 X_jU, X_k -\frac12 X_kU]\\
    &=
 [X_j,X_k]-\frac12[X_j,X_kU] -\frac12 [X_jU, X_k]\\
 &=
  [X_j,X_k] -\frac12\left(X_j X_kU- X_kX_jU\right)
  =[X_j,X_k] -\frac12 [X_j,X_k] U \equiv V_{j,k}.
\end{split}
\]
That is, we have the following commutator property
 \[\tag{**}
 [V_j,V_k] 
  =[X_j,X_k] -\frac12 [X_j,X_k] U \equiv V_{j,k}.
\]
Thus, if $X_j$, $j=1,..,k$ generate a nilpotent Lie algebra, then, the corresponding fields $V_j\equiv V(X_j,U)$,  $j=1,..,k$, generate a nilpotent Lie algebra of the same structure. \\
Define the following operators
\[\mathfrak{L}\equiv -\sum_{j=1,..,k} V_j^2\]
and 
\[\Theta\equiv -\sum_{\alpha} V_\alpha^2,\]
where $\alpha\equiv (j_1,j_2,..,j_l)$, $l\in\mathbb{N}$.
Define also
\[\Lambda^s\equiv (1+\Theta)^{s/2}.\]

\noindent In the above setup, we have the question of equivalence of norms given by the following expressions
$\|\mathfrak{L}^mf\|_2^2$ 
and
$\sum_{j_1,..,j_n\leq k}\|V_{j_1}..V_{j_n}f\|_2^2$ .

\noindent For example, in case $m=1$, we have
\[\begin{split}
   \|\mathfrak{L} f\|_2^2 &= \sum_{j,l\leq k}(V_j^2f, V_l^2f)=-\sum_{j,l\leq k}(V_j f, V_jV_l^2f)\\
   &=\sum_{j,l\leq k}(V_lV_j f, V_jV_l f) -\sum_{j,l\leq k}(V_j f, [V_j,V_l]V_lf)\\
   &=\sum_{j,l\leq k}(V_j V_l f, V_jV_l f) +\sum_{j,l\leq k}([V_l,V_j] f, V_jV_l f)-\sum_{j,l\leq k}(V_j f, [V_j,V_l]V_lf)\\
   &=\sum_{j,l\leq k}(V_j V_l f, V_jV_l f) +\frac{1}{2}\sum_{j,l\leq k}\left(([V_l,V_j] f, V_jV_l f) + ([V_j,V_l] f, V_lV_j f)\right)\\
   &-\sum_{j,l\leq k}(V_j f, [V_j,V_l]V_lf).\\
\end{split}\]
Since
\[\begin{split}
&\frac{1}{2}\sum_{j,l\leq k}\left(([V_l,V_j] f, V_jV_l f) + ([V_j,V_l] f, V_lV_j f)\right) = \frac{1}{2}\sum_{j,l\leq k}\left(([V_l,V_j] f, V_jV_l f) - ([V_l,V_j] f, V_lV_j f)\right)\\
&=-\frac{1}{2}\sum_{j,l\leq k}   ([V_l,V_j] f, [V_l,V_j]f), 
\end{split}\]
we get 
 \[
  \|\mathfrak{L} f\|_2^2 +\frac{1}{2}\sum_{j,l\leq k}   ([V_l,V_j] f, [V_l,V_j]f)  = \sum_{j,l\leq k}(V_j V_l f, V_jV_l f) -\sum_{j,l\leq k}(V_j f, [V_j,V_l]V_lf).
 \]
  \textbf{The case of nilpotent Lie algebras of order two} \\
 If the algebra is of order two, then
 $[V_j,V_l]$ would be in the center, and we would have
\[\begin{split}
-\sum_{j,l\leq k}(V_j f, [V_j,V_l]V_lf)&=
 \sum_{j,l\leq k}(V_lV_j f, [V_j,V_l]f)=
 \frac12\sum_{j,l\leq k}\left((V_lV_j f, [V_j,V_l]f) +(V_jV_lf, [V_l,V_j]f)\right)\\
 &=
 \frac12\sum_{j,l\leq k}\left((V_lV_j f, [V_j,V_l]f) - (V_jV_lf, [V_j,V_l]f)\right)
 \\
 &=
 -\frac12\sum_{j,l\leq k}  ([V_j,V_l]f, [V_j,V_l]f) .
 \end{split}\]

We have the following result:
 \begin{prop}
 If the nilpotent Lie algebra is of second order, then
 \[
  \|\mathfrak{L} f\|_2^2 + \sum_{j,l\leq k}   ([V_l,V_j] f, [V_l,V_j]f)  = \sum_{j,l\leq k}(V_j V_l f, V_jV_l f),  
 \]
and hence we have the following relation of seminorms 
 \[
  \|\mathfrak{L} f\|_2^2   \leq  \sum_{j,l\leq k}(V_j V_l f, V_jV_l f) .
 \]
 \end{prop}

 Next, we remark that
 \[
  \|\mathfrak{L} f\|_2^2= (f,\mathfrak{L}^2f),
 \]
and, using the fact that  $\mathfrak{L}$ is a non-negative self-adjoint operator,  by spectral theory, we have that
\[ (f,\mathfrak{L}f)\leq (f,\mathfrak{L}^2f)+\|f\|^2.\]
Recalling the definition
\[ \Theta\equiv -\sum_\alpha V_\alpha^2 =  \frak{L}- \sum_{j,l\leq k}[V_j, V_l]^2,
\]
we consequently get
\begin{equation} \label{Theta.1}
    (f, \Theta f)\leq \sum_{j,l\leq k}(V_j V_l f, V_jV_l f) +\|f\|^2.
\end{equation}  
Whence, we conclude with the following result for a Sobolev-type norm associated to the operator $\Lambda_1 \equiv (1+\Theta)^\frac12$:
\begin{prop}
If the nilpotent Lie algebra is of second order, then
\[ \|\Lambda_1 f\|_2^2\leq 2\left(\sum_{j,l\leq k}(V_j V_l f, V_jV_l f) +\|f\|_2^2\right).\]
\end{prop}
This is a regularity statement, in the sense that the right-hand side includes only the generators of the algebra, while the left-hand side includes all directions in the second strata.\\

\noindent \textbf{Explicit representations}: \\
We will use the above results for a discussion of comparison of seminorms defined in terms of original fields $X_j$,  $j=1,..,k.$  We have
	\[ \sum_{j,l\leq k}(V_j V_l f, V_jV_l f) 	=  \sum_{j,l\leq k}\mu((X_j-\frac12X_jU)(X_l-\frac12X_lU) f, (X_j-\frac12X_jU)(X_l-\frac12X_lU) f), 
	\]
	with  
	\[V_j V_l f=  (X_j-\frac12X_jU)(X_l-\frac12X_lU)f= X_jX_lf-\frac12(X_jU)X_lf -\frac12(X_lU)X_jf  +\frac{1}{4}(X_j U) (X_l U)f-\frac12(X_jX_lU)f.
	\]
	Hence,
	\begin{equation} \label{Theta.2}
 \begin{split} \sum_{j,l\leq k}(V_j V_l f, V_jV_l f) 	&\leq 
	C	\sum_{j,l\leq k}\mu (X_jX_lf,X_jX_lf) + C	\sum_{j,l\leq k}\mu\left( |X_lf|^2
\left( \frac14 (X_jU)^2\right)\right) \\
&\qquad + C	\sum_{j,l\leq k} \frac1{16} \mu\left( 
 (X_lU X_jU)^2f^2\right)
+ C	\sum_{j,l\leq k} \frac14 \mu\left( 
 (X_lX_jU)^2f^2\right)\\
 &= C\mu|\nabla\nabla f|^2 +C \mu \left( |\nabla f|^2\left(\frac14|\nabla U|^2  \right)\right) +
 C \mu \left(\frac1{16}  \left( |\nabla U|^2  \right)^2 f^2\right)
 +C  \mu \left(  
 \frac14|\nabla\nabla U|^2   f^2\right),
\end{split}	
\end{equation} 
where $\nabla$ and $\Delta$ denote the sub-gradient and the sub-Laplacian, respectively, with respect to the generating fields $X_j$, $j=1,..,k$. 
	For the second term on the right-hand side,
 with
	\[\mathcal{V}\equiv \frac14|\nabla U|^2 -\frac12\Delta U,\] 
 by the usual $U$-bound, \cite{HZ}, we get
	\begin{equation} \label{Theta.3} \mu \left( |\nabla f|^2\left(\frac14|\nabla U|^2  \right)\right) = \mu \left( |\nabla f|^2 \mathcal{V} \right)  + \mu \left( |\nabla f|^2 \frac{1}{2}\Delta{U} \right)\leq C_1\mu |\nabla \nabla f|^2 + C_2 \mu |\nabla f|^2+  \mu \left( |\nabla f|^2 \frac{1}{2}\Delta{U} \right).
	\end{equation}
	The last term on the right-hand side contains the factor $\frac{1}{2}\Delta{U},$ which is potentially of lower order and could be bounded by $\tilde{C}\; \mu|\nabla\nabla f|^2,$ for some constant $\tilde {C} \in (0,\infty )$.
	To bound the last term on the right-hand side, generally, we need some additional assumptions 
	on $U$. (As shown in Section \ref{sec6}, generally Adams' regularity condition may not hold. )\\
 
	\noindent Next, we note that
	\begin{equation} \label{Theta.4} 
 \begin{split}
	    (f,\Theta f) &= (f,\mathfrak{L} f) +\|(Z-\frac12ZU)f\|^2 = (f, (-Lf) )
	- \mu(f^2\mathcal{V})   +\|(Z-\frac12ZU)f\|^2\\
 &= (f, (-L-L_Z)f)  
	- \mu(f^2\mathcal{V})   - \mu(f^2\mathcal{V}_Z), 
	\end{split}
 \end{equation}
	where $L$ and $L_Z$ denote the Dirichlet operator associated to the sub-gradient and $Z$ fields (in the second strata), respectively, and $\mathcal{V}$ and 
 $\mathcal{V}_Z\equiv \frac{1}{4} |\nabla_Z U|^2-\frac12\Delta_Z U$ denote their  corresponding potentials, respectively.
 Combining \eqref{Theta.1}-\eqref{Theta.4}, we obtain the following bound 
\[(f, (-L-L_Z)f)\]
\begin{equation} \label{Theta.5}
 \leq 
(C+C C_1)\mu|\nabla\nabla f|^2 +C C_2\mu|\nabla f|^2 
+ \frac{C}4 \mu \left( f^2\left(|\nabla\nabla U|^2  \right)\right) 
+ \frac{C}{16} \mu \left( f^2\left(|\nabla U|^2  \right)^2\right)
+ \frac{C }2 \mu \left(  
  |\nabla f|^2\Delta U \right)
	+ \mu(f^2\mathcal{V})   + \mu(f^2\mathcal{V}_Z)+\|f\|^2.
 \end{equation}
 Hence, we arrive at the following result, which provides a bound on a Sobolev-type norm by norms involving only the (iterated) sub-gradients.

 \begin{thm}
 Let $\mathfrak{G}$ be a nilpotent Lie algebra is of second order. Suppose the interaction $U$ satisfies the following quadratic form bound
	\[ \begin{split}
 \frac{C}{16} \mu \left( f^2\left(|\nabla U|^2  \right)^2\right)+
 \frac{C }4\mu \left( f^2\left(|\nabla\nabla U|^2  \right)\right) + \frac{C }2 \mu \left(  
  |\nabla f|^2\Delta U \right)
	+ \mu(f^2\mathcal{V})   + \mu(f^2\mathcal{V}_Z) 
	\leq   C' \left( \mu |\nabla\nabla f|^2 +\mu |\nabla f|^2+\mu f^2 \right),
 \end{split}
	\]
 with some constant $C'\in(0,\infty)$ independent of $f$.
 Then, the following global regularity bound holds 
 \[
 \|(1-L-L_Z)^\frac12 f\|_2^2\leq
 D \left( \mu|\nabla\nabla f|^2 + \mu|\nabla f|^2 +\mu f^2\right)
 \]
 with a constant $D\in(0,\infty)$ independent of $f$.
 \end{thm}
 
 \begin{example}
 Consider the Heisenberg group $\mathbb{H}_1$ with  generating fields
\[\begin{split}
    X&=\partial_x +\frac12 y\partial_z;  \\
Y&= \partial_y -\frac12 x\partial_z\\
\end{split}
\]
and the corresponding Kaplan norm 
\[
N \equiv \left(\left(x^2+y^2\right)^2 +16 z^2\right)^{\frac{1}{4}}.
\]
With $\kappa\in(0,\infty)$, let 
\[ U\equiv N^\kappa \]
From \cite{HZ} it satisfies the following U-Bound
\[\mu f^2\mathcal{V}\leq \mu|\nabla f|^2.\]
We have shown in Section \ref{sec6} that this interaction is not Adams' regular. \\
For quadratic form estimate with potential $\mathcal{V}_Z$,
with $Z \equiv [X,Y]= \partial_z\equiv Z$, we note that
\[\begin{split}
    ZN &=8 N^{-3} z, \\
    Z^2N &=8 Z\left(N^{-3} z \right) = 
    -3\cdot 2^6 N^{-7}  z^2 + 8  N^{-3}.
\end{split}
\]
Thus
\[ \begin{split}
\mathcal{V}_Z\equiv \frac14|ZU|^2-\frac12 Z^2 U 
&= \frac14(U'(N))^2|ZN|^2-\frac12 Z \left( U'(N)ZN\right)\\
&= \frac12\left(\frac12(U'(N))^2 -U^"(N) \right)|ZN|^2 -\frac12   U'(N)Z^2 N \\
&= \frac12\left(\frac{\kappa^2}{2} N^{2(\kappa-1)} -\kappa(\kappa-1)N^{\kappa-2} \right) 64 N^{-6} z^2 -\frac12   \kappa N^{\kappa-1} \left(-3\cdot 2^{6} N^{-7}  z^2 + 8  N^{-3}\right)\\
&= 32\kappa N^{\kappa-4}\left(\frac{\kappa}{2} N^{\kappa} - (\kappa-1) \right)   N^{-4} z^2 -4   \kappa N^{\kappa-4} \left(-3\cdot 2^{3} N^{-4}  z^2 + 1   \right)\\
&= 4\kappa N^{\kappa-4}\left[8\left(\frac{\kappa}{2} N^{\kappa} - (\kappa-1) \right)   N^{-4} z^2 - \left(-3\cdot 2^{3} N^{-4}  z^2 + 1   \right)\right]\\
&= 16\kappa N^{\kappa-4}\left[\left( \kappa  N^{\kappa} +3 \cdot 2  - 2(\kappa-1) \right)   N^{-4} z^2   - \frac14   \right]
\end{split}
\]
Using $z^2=\frac1{16}\left(N^4-|\mathbf{x}|^4\right)$, we can represent the above as follows
\[ \begin{split}
\mathcal{V}_Z\equiv \frac14|ZU|^2-\frac12 Z^2 U 
&= 16\kappa N^{\kappa-4}\left[\left( \kappa  N^{\kappa} +  3\cdot 2  - 2(\kappa-1) \right)   N^{-4} z^2   - \frac14   \right]\\
&=  \kappa N^{\kappa-4}\left[\left( \kappa  N^{\kappa} + 3\cdot 2  - 2(\kappa-1) \right)    \left(1 -\frac{ |\mathbf{x}|^4}{N^{4}}\right)  - 4   \right]
%
\end{split}
\]

For $q\geq 2$, in \cite{BDZ} we got the following 
\[
 \int \frac{U'(N)}{N^2}|f|^q d\mu\leq C |\nabla f|^q d\mu+D |f|^qd\mu   
\]
In case when $U(N)\equiv N^\kappa$, we have $\frac{U'(N)}{N^2}=\kappa N^{\kappa-3},$ which is not sufficient for us. 
Fortunately, using a technique employing Hardy's inequality,  \cite{YaoThPhD}, one can improve that to obtain
\[
 \int \frac{U'(N)}{N}|f|^qd\mu\leq C |\nabla f|^qd\mu+D |f|^qd\mu   .
\]
In particular, in the case when $U(N)\equiv N^\kappa$, we have $\frac{U'(N)}{N}=\kappa N^{\kappa-2}$.\\

Thus, iterating the quadratic form bound, we arrive at the following desired regularity estimate
\[
\int\left(  \frac14|ZU|^2-\frac12 Z^2 U \right) f^2\leq
\tilde A \mu \left(|\nabla\nabla f|^2 +|\nabla f|^2 +f^2\right),
\]
\end{example}

Hence we obtain the following result:

\begin{thm}
Let $\mathfrak{G}=\mathbb{H}_1$. Suppose the interaction $U=N^\kappa$, with $\kappa>2$.
 Then the following global regularity bound is true.
 \[ \|\Lambda_1 f\|_2^2\leq 2\left(\sum_{j,l\leq k}(V_j V_l f, V_jV_l f) +\|f\|_2^2\right).\]
\end{thm}


\subsection{Towards Global Regularity II} \label{sec8.1}
In this section, we continue the discussion of the global regularity problem in a general setup described at the beginning of the previous section.\
That is, given a probability measure $d\mu\equiv e^{-U}d\lambda$, we consider
the fields
\[
V_j\equiv V(X_j,U)\equiv X_j -\frac12 X_jU
\]
associated with a first-order differential operator $X_j$.
We have
\[\tag{*}
\mu (fV_jg)= - \mu ((V_jf)g)
\]
We also have
\[\begin{split}
    [V_j,V_k]&= [X_j -\frac12 X_jU, X_k -\frac12 X_kU]\\
    &=
 [X_j,X_k]-\frac12[X_j,X_kU] -\frac12 [X_jU, X_k]\\
 &=
  [X_j,X_k] -\frac12\left(X_j X_kU- X_kX_jU\right)
  =[X_j,X_k] -\frac12 [X_j,X_k] U \equiv V_{j,k}
\end{split}
\]
That is, we have the following commutator properties
 \[\tag{**}
 [V_j,V_i] 
  =[X_j,X_i] -\frac12 [X_j,X_i] U \equiv V_{j,i}
\]
Thus, if $X_j$, $j=1,..,k$ generate a nilpotent Lie algebra, then the corresponding fields $V_j\equiv V(X_j,U)$,  $j=1,..,k$, generate a nilpotent Lie algebra of the same structure.
If $U$ is $n$-time differentiable, the corresponding field obtained by an $n^{th}$ order commutator is defined as follows: 
for an ordered $n$-tuple $\alpha\equiv (j_1,..,j_n)$, $j_l\leq k$ for $l=1,..,n$,
we set 
\[\tag{***}
V_{j_1,..,j_n}\equiv [V_{j_n},..,[V_{j_2},V_{j_1}]..] \equiv
[X_{j_n},..,[X_{j_2},X_{j_1}]..] -\frac12 [X_{j_n},..,[X_{j_2},X_{j_1}]..] U  .
\]
\\
Define the following operators
\[\mathfrak{L}\equiv -\sum_{j=1,..,k} V_j^2,\]
and, for $\lambda\in(0,\infty),$
\[\mathfrak{R}_{\lambda}\equiv (\lambda +\mathfrak{L})^{-1} .\]
For simplicity, we set
\[\mathfrak{R}\equiv (1+\mathfrak{L})^{-1}=\mathfrak{R}_{\lambda=1}.\]
 We note that
 \[ \|\mathfrak{R}_{\lambda}\| \leq \frac{1}{\lambda}.\]
 We are after the following result, whose proof is completed below for the case of second-order nilpotent Lie groups.
\begin{thm}
For any $s\in\mathbb{N}$, there exists $\tilde C\in(0,\infty)$ such that
\[\begin{split}\|(1+\mathfrak{L})^{\frac{s}{2}} V_l (1+\mathfrak{L})^{-\frac{s}{2}-\frac12}\tilde f\|_2 &\leq 
\tilde C\|\tilde f\|_2.
\end{split}\]
\end{thm}
Note that from the definition of the norm, the integration by parts formula for the field $V_l,$ and the definition of $\mathfrak{L}$, we have the case $s=0:$
\[
\|V_l\mathfrak{R}^{\frac12}\tilde f\|_2\leq  \|\tilde f\|_2 .
\]
Thus, to show the lemma, the following bound would be sufficient
\[\begin{split}
\|[(1+\mathfrak{L})^{\frac{s}{2}} ,V_l] \mathfrak{R}^{\frac{s}{2}+\frac12}\tilde f\|_2 &\leq 
\tilde C\|\tilde f\|_2.
\end{split}\]
For $s=2n$, by an inductive use of the operator identity 
\[[A^n,B]=[A,B]A^{n-1}+A[A^{n-1},B],\]
one gets, with $ad_A(B)\equiv [A,B]$,
\[
[A^n,B]= \sum_{m=1}^n\binom{n}{m}ad_A^m(B)A^{n-m}
.
\]
Hence, for $\frac{s}{2}=n$, we have
\[\begin{split}
 [(1+\mathfrak{L})^{n} ,V_l]  &=
 \sum_{m=1}^n (-1)^m\binom{n}{m}
ad_{\mathfrak{L}}^m(V_l)(1+\mathfrak{L})^{n-m} ,
\end{split}\]
which implies
 \begin{equation}\label{Kn}\begin{split}
\|[(1+\mathfrak{L})^{n} ,V_l] (1+\mathfrak{L})^{-n-\frac12}\tilde f\|_2 &\leq 
 \sum_{m=1}^n  \binom{n}{m}
\|ad_{\mathfrak{L}}^m(V_l)(1+\mathfrak{L})^{n-m}  (1+\mathfrak{L})^{-n-\frac12}\tilde f\|_2 \\
&=\sum_{m=1}^n  \binom{n}{m}
\|ad_{\mathfrak{L}}^m(V_l)\mathfrak{R}^{m} \mathfrak{R}^{\frac12}\tilde f\|_2 .
\end{split}
 \end{equation}
Note that the contributions to the multiple commutator will depend on the order of the nilpotent group. We observe that using identity $[A^2,B]=\{A,[A,B]\}$, we get
\[
ad_{\mathfrak{L}}(V_l)=\sum_{j_1\leq k}
\{V_{j_1}, V_{lj_1}\},
\]
with the curly bracket denoting the anticommutator. Then, using the Leibniz rule
\[
[A,\{B,C\}]= \{[A,B],C\}+\{B,[A,C]\},
\]
we get
\[\begin{split}
ad_{\mathfrak{L}}^2(V_l)&=\sum_{j_1j_2\leq k}
\{V_{j_2},[V_{j_2},\{V_{j_1}, V_{lj_1}\}
] \} \\
&= \sum_{j_1j_2\leq k}
\{V_{j_2}, \{V_{j_1j_2}, V_{lj_1}\}+\{V_{j_1}, V_{lj_1j_2}\}\}.
\end{split}\]
For example, for a nilpotent Lie group of order two, the last formula will simplify as follows
\[ 
ad_{\mathfrak{L}}^2(V_l) = \sum_{j_1j_2\leq k}
\{V_{j_2}, \{V_{j_1j_2}, V_{lj_1}\} \}
=4 \sum_{j_1j_2\leq k}
 V_{j_2} V_{j_1j_2}  V_{lj_1} 
\]
where in the last step we used the fact that second-order fields are in the center of the algebra.
Consequently, we get
\[ 
ad_{\mathfrak{L}}^3(V_l) = \sum_{j_1j_2\leq k} ad_{\mathfrak{L}}\left(
\{V_{j_2}, \{V_{j_1j_2}, V_{lj_1}\} \}\right)
=4 \sum_{j_1j_2j_3\leq k}\{V_{j_3}, [ V_{j_3}, V_{j_2} V_{j_1j_2}  V_{lj_1}]
 \},
\]
and since higher than second-order fields are zero, we get
\[ 
ad_{\mathfrak{L}}^3(V_l) 
=4 \sum_{j_1j_2j_3\leq k}\{V_{j_3},  V_{j_2j_3} V_{j_1j_2}  V_{lj_1} 
 \}
 = 8 \sum_{j_1j_2j_3\leq k} V_{j_3}  V_{j_2j_3} V_{j_1j_2}  V_{lj_1} .
\]
 Hence, by induction in the second order case, for any $n\in\mathbb{N}$ we get
 \begin{equation} \label{KO2}
   ad_{\mathfrak{L}}^n(V_l) 
 = 2^n \sum_{j_1j_2j_3..j_n\leq k} V_{j_{n}}V_{j_{n}j_{n-1}} .. V_{j_2j_3} V_{j_1j_2}  V_{lj_1}.
 \end{equation}
In terms of graphs, we can interpret this by a property that a vertex may be connected to at most two others.
Using the last expression \eqref{KO2}, we can bound each term from \eqref{Kn} as follows
\[
\|ad_{\mathfrak{L}}^m(V_l)\mathfrak{R}^{m} \mathfrak{R}^{\frac12}\tilde f\|_2\leq 
2^m \sum_{j_1j_2j_3..j_m\leq k} \|V_{j_{m}}V_{j_{m}j_{m-1}} .. V_{j_2j_3} V_{j_1j_2}  V_{lj_1}\mathfrak{R}^{m} \mathfrak{R}^{\frac12}\tilde f\|_2.
\]
If the operator norm $\max \|V_{ji}\mathfrak{R}\|\equiv D\in(0,\infty)$  , 
we end up with the estimate 
\[
\|ad_{\mathfrak{L}}^m(V_l)\mathfrak{R}^{m} \mathfrak{R}^{\frac12}\tilde f\|_2\leq 
2^m D^m k^{m+1} \max \|V_{j} \mathfrak{R}^{\frac12}\tilde f\|_2.
\]
Consequently, we obtain bound on \eqref{Kn} as follows
\begin{equation}\label{KnFin} 
\|[(1+\mathfrak{L})^{n} ,V_l] (1+\mathfrak{L})^{-n-\frac12}\tilde f\|_2  \leq 
 \sum_{m=1}^n \binom{n}{m}
2^m D^m k^{m+1}   \max\|V_{j} \mathfrak{R}^{\frac12}\tilde f\|_2 =k(1+2Dk)^n\max\|V_{j} \mathfrak{R}^{\frac12}\tilde f\|_2,
 \end{equation}
with the right-hand side being finite.
This completes the discussion of the second-order case.
\qed

\begin{rem}
For the third-order case, we have
\[\begin{split}
ad_{\mathfrak{L}}^3(V_l)
&= \sum_{j_1j_2\leq k}
ad_{\mathfrak{L}}\left(\{V_{j_2}, \{V_{j_1j_2}, V_{lj_1}\}+\{V_{j_1}, V_{lj_1j_2}\}\}\right)\\
&= \sum_{j_1j_2j_3\leq k}
 \left(\{V_{j_3},\{V_{j_2j_3}, \{V_{j_1j_2}, V_{lj_1}\}\} +\{V_{j_3},\{V_{j_2}, \{V_{j_1j_2j_3}, V_{lj_1}\}\}
 +\{V_{j_3},\{V_{j_2}, \{V_{j_1j_2}, V_{lj_1j_3}\}\}\right)\\
&+\sum_{j_1j_2j_3\leq k}
\left(\{V_{j_3},\{V_{j_2j_3},\{V_{j_1}, V_{lj_1j_2}\}\}\} +\{V_{j_3},\{V_{j_2},\{V_{j_1j_3}, V_{lj_1j_2}\}\}\}
+\{V_{j_3},\{V_{j_2},\{V_{j_1}, V_{lj_1j_2j_3}\}\}\}\right).
\end{split}\]
Since the fourth-order fields are zero and the third-order fields are in the center, in the third-order case  this simplifies as follows
\[\begin{split}
ad_{\mathfrak{L}}^3(V_l)
&= \sum_{j_1j_2j_3\leq k}
 \left(\{V_{j_3},\{V_{j_2j_3}, \{V_{j_1j_2}, V_{lj_1}\}\} +\{V_{j_3},\{V_{j_2}, \{V_{j_1j_2j_3}, V_{lj_1}\}\}
 +\{V_{j_3},\{V_{j_2}, \{V_{j_1j_2}, V_{lj_1j_3}\}\}\right)\\
&+\sum_{j_1j_2j_3\leq k}
\left(\{V_{j_3},\{V_{j_2j_3},\{V_{j_1}, V_{lj_1j_2}\}\}\} +\{V_{j_3},\{V_{j_2},\{V_{j_1j_3}, V_{lj_1j_2}\}\}\}
\right)\\
&= \sum_{j_1j_2j_3\leq k}
 \left(\{V_{j_3},\{V_{j_2j_3}, \{V_{j_1j_2}, V_{lj_1}\}\} +2\{V_{j_3},\{V_{j_2},  V_{lj_1}\}\}V_{j_1j_2j_3}
 +2^2\{V_{j_3}, V_{j_2} \} V_{j_1j_2}  V_{lj_1j_3}\right)\\
&+\sum_{j_1j_2j_3\leq k}
2^2\left( \{V_{j_3}, V_{j_1} \}V_{j_2j_3}V_{lj_1j_2} +\{V_{j_3}, V_{j_2}\} V_{j_1j_3}  V_{lj_1j_2} \right).\\
\end{split}\]
In terms of graphs, we admit vertices of at most third order.
\end{rem}
\section{Hardy's Inequality for Probability Measures} \label{sec9}

Let $U\equiv U(N)$ and let $d\mu\equiv e^{-U}d\lambda$. Let $\nabla\equiv (X_j)_{j=1,..,n}$ be a sub-gradient associated to a nilpotent Lie group.
In this context, we have the following Hardy's inequality \cite{RS},
\[
\int \frac{f^2}{|\mathbf{x}|^2}d\lambda\leq C \int |\nabla f|^2 d\lambda.
\]
Substituting $f e^{-\frac12 U}$ instead of $f$, we get
\[
\int \frac{f^2}{|\mathbf{x}|^2}d\mu\leq C \int |\nabla f-\frac12f\nabla U |^2 d\mu \leq
  2C \int |\nabla f|^2 d\mu+
 \frac{C}{2} \int f^2|\nabla U |^2 d\mu .
\]
The corresponding $U$-bound is as follows
\[
\int f^2 \left( \frac14 |\nabla U|^2 -\frac12 \Delta U\right) d\mu \leq \int 
|\nabla f|^2d\mu.
\]
In case of Heisenberg group,  for $U=U(N)$, since $\Delta N= (n+2m-1)\frac{|\nabla N|^2}{N}=(n+2m-1)\frac{|x|^2}{N^3}$, we have
\[\begin{split}
 \int f^2 \left( \frac14 |\nabla U|^2 -\frac12 \Delta U\right) d\mu &= \int f^2 \left( \frac14 | U'|^2|\nabla N|^2 -\frac12   U"|\nabla N|^2 -\frac12   U'\Delta N\right) d\mu \\ 
 &=\int f^2 \left( \left( \frac14 | U'|^2  -\frac12   U"\right) \frac{|\mathbf{x}|^2}{N^2} -  {\frac{n+2m-1}{2}}   U'\frac{|\mathbf{x}|^2}{N^3}\right) d\mu. 
\end{split}\]

 \noindent In this class of examples, if $|U'|^2$ grows faster to infinity than $U"$, with some $D\in(0,\infty),$ we have
 \[ |\nabla U|^2\leq D\left(1+ \frac14|\nabla U|^2-\frac12 \Delta U\right).\]
 This implies the following Hardy-type inequality with a constant $\tilde  C\in(0,\infty)$
\[ 
\int\frac{f^2}{|\mathbf{x}|^2}d\mu \leq  \tilde  C \int |\nabla f|^2 d\mu + \int f^2 d\mu.\]
 Using this Hardy's inequality and the $U$-bound inequality, with a constant $D\geq 0$ such that
  
 \[W\equiv \left( \left( \frac14 | U'|^2  -\frac12   U"\right)  -\frac{n+2m-1}{2}   U'\frac{1}{N}\right) + D  \geq 0,\] 
 we have 
 \[
 \int f^2\left(\frac{1}{N^2} \cdot \frac{N^2}{|\mathbf
 {x}|^2} +W \frac{|\mathbf{x}|^2}{N^2}  \right) d\mu \leq \tilde C \int |\nabla f|^2d\mu +\tilde D \int f^2.
 \]
 Hence, (using \cite{YaoThPhD} minimization with respect to the parameter $\frac{|\mathbf{x}|^2}{N^2}$), we get
 \[
 \int f^2 \frac{W^\frac12}{N} d\mu \leq 
 \tilde C \int |\nabla f|^2d\mu +\tilde D \int f^2.
 \]
 If $U=N^\kappa$, we get 
 \[
 \frac{W^\frac12}{N} \sim N^{\kappa - 2},
 \]
 which is better than the bound in \cite{BDZ}
 where the left-hand side contains a multiplier
 $U'/N^2\sim  N^{\kappa-3}$.

\section{Examples of Regular and Nonregular Potentials} \label{sec10}
\begin{example} \label{E.g.1}
 For $\varepsilon\in(0,1)$, let 
\[U(\mathbf{x})\equiv  r^{\frac{1}{1-\varepsilon}\left(1+\varepsilon\cos(\varphi) \right)},\] and let
\[d\mu \equiv \frac{e^{-U}}{\int e^{-U}d\lambda}d\lambda .\]
In particular, for $\varphi=0$, we have 
\[ U(\mathbf{x})_{|\varphi=0}=  r^{\frac{1+\varepsilon }{1-\varepsilon}},\]
with $(1+\varepsilon)/(1-\varepsilon)\geq 2$ for $\varepsilon \geq 1/3$,
while for $\varphi=\pi$,
we have
\[ U(\mathbf{x})_{|\varphi=\pi}=  r.\]
The gradient of $U$ in polar coordinates is as follows
\[\begin{split} 
\nabla U &= \left(\mathbf{e}_{r} \partial_r+
\mathbf{e}_\varphi \frac1{r}\partial_\varphi \right) U
=
\nabla \exp\{\frac{1}{1-\varepsilon}(1+\varepsilon \cos(\varphi))\log r\} \\
&=  \mathbf{e}_{r} U \left( \frac{1}{1-\varepsilon}(1+\varepsilon \cos(\varphi))\frac{1}{r} \right)  - \mathbf{e}_\varphi U\frac{\varepsilon }{1-\varepsilon} \sin(\varphi) \frac{\log r}{r} .
\end{split}
\]
Hence,
\[\begin{split}
|\nabla U|^2  &=  U^2 \left( \frac{1}{1-\varepsilon}(1+\varepsilon \cos(\varphi))\frac{1}{r} \right)^2  +U^2\left(\frac{\varepsilon }{1-\varepsilon} \sin(\varphi) \frac{\log r}{r} \right)^2\\
&=  \frac{1}{(1-\varepsilon)^2}\frac{1}{r^2}U^2 \left(   (1+\varepsilon \cos(\varphi))^2  +  \varepsilon^2   \sin^2(\varphi) (\log r)^2 \right) .
\end{split}
\]
In particular, for $\varphi=\pi$, we have
\[|\nabla U|^2 = \frac{1}{r^2}U^2= r^{\frac{2}{1-\varepsilon}\left(1+\varepsilon\cos(\varphi) \right) - 2} =1 \]
For the Laplacian, we have that
\[
\Delta U  =   
 \partial_r^2 U + \frac{1}{r}\partial_r U
 +\frac{1}{r^2} 
 \partial_\varphi^2 U 
\]
Hence,
\[
\begin{split} 
\Delta U  &=   
 U \left( \frac{1}{1-\varepsilon}(1+\varepsilon \cos(\varphi))\frac{1}{r} \right)^2
 -U \left( \frac{1}{1-\varepsilon}(1+\varepsilon \cos(\varphi))\frac{1}{r^2} \right)
 \\
 &+ \frac{1}{r}U \left( \frac{1}{1-\varepsilon}(1+\varepsilon \cos(\varphi))\frac{1}{r} \right) \\
 &-\frac{1}{r^2} 
 \partial_\varphi  \left(U\frac{\varepsilon }{1-\varepsilon} \sin(\varphi)  \log r  \right) \\
 &= U \left( \frac{1}{1-\varepsilon}(1+\varepsilon \cos(\varphi))\frac{1}{r} \right)^2
 -U \left( \frac{1}{1-\varepsilon}(1+\varepsilon \cos(\varphi))\frac{1}{r^2} \right)
 \\
 &+ \frac{1}{r}U \left( \frac{1}{1-\varepsilon}(1+\varepsilon \cos(\varphi))\frac{1}{r} \right) \\
 &+\frac{1}{r^2} 
  U\left(\frac{\varepsilon }{1-\varepsilon} \sin(\varphi)  \log r  \right)^2 
 -\frac{1}{r^2} 
  U\frac{\varepsilon }{1-\varepsilon} \cos(\varphi)  \log r    \\
  & \\
  &=  \frac{1}{r^2}U\left(
  \left( \frac{1}{1-\varepsilon}(1+\varepsilon \cos(\varphi))\right)^2
 + 
 \left(\frac{\varepsilon }{1-\varepsilon} \sin(\varphi) \frac{\log r}{r} \right)^2 
 - \frac{\varepsilon }{1-\varepsilon} \cos(\varphi) \frac{\log r}{r}  \right).
\end{split}
\]
Thus,
\[\Delta U =\frac{1}{(1-\varepsilon)^2}\frac{1}{r^2}U\left(
  \left( 1+\varepsilon \cos(\varphi) \right)^2
 + 
 \left( \varepsilon  \sin(\varphi)  \log r \right)^2 
 -  \varepsilon  (1-\varepsilon)\cos(\varphi)  \log r   \right) .
\]
It is interesting to see the values for particular $\phi$. In particular, at $\varphi=\pi$, we have
\[\begin{split}
    \Delta U &= \frac{1}{r} \left(
1
+ \frac{\varepsilon}{1-\varepsilon}   \log r  \right) .
\end{split}
\]
Thus, for $\varphi=\pi,$ we have
\[
\frac{1}{4}|\nabla U|^2 - \frac12 \Delta U = 
\frac12\left(\frac{1}{2} -  \frac{1}{r} \left(
1
+ \frac{\varepsilon}{1-\varepsilon} \log r   \right)\right) ,
\]
which is strictly positive for large $r$, but does not grow to infinity with $r\to\infty$.\\
On the other hand, for $\varphi=0$, we have
\[
\left(\frac{1}{4}|\nabla U|^2 - \frac12 \Delta U \right)_{|\varphi=0}= \frac{(\varepsilon +1)^2 r^{-\frac{4 \varepsilon }{\varepsilon -1}}-2 r^{-\frac{2}{\varepsilon -1}-3} \left((\varepsilon +1)^2+(\varepsilon -1) \varepsilon  \log (r)\right)}{4 (\varepsilon -1)^2},
\]
which is positive for $\epsilon<1$ and $r$ large enough,
and in fact growing to infinity as $r\to\infty$.\\
Finally, for $\varphi=\pm \pi/2$, we have
\[
\left(\frac{1}{4}|\nabla U|^2 - \frac12 \Delta U \right)_{|\varphi=\pm\pi/2}= 
-\frac{r^{-\frac{2 \epsilon }{\epsilon -1}} \left(2 r^{\frac{1}{\epsilon -1}}-1\right) \left(\epsilon ^2 \log ^2(r)+1\right)}{4 (\epsilon -1)^2},
\] 
which is positive for $\epsilon<1$ and $r$ large enough, and diverges to infinity with $r\to \infty$. \\

In general, we have
\[\begin{split}
\frac{1}{4}|\nabla U|^2 \pm \frac12 \Delta U &=  
\frac12  \frac{1}{(1-\varepsilon)^2}\frac{1}{r^2}U \left[
\frac12  U \left(  (1+\varepsilon \cos(\varphi))^2  +  \varepsilon^2   \sin^2(\varphi) (\log r)^2 \right) \right.\\
& \left. \pm \left(
  \left( 1+\varepsilon \cos(\varphi) \right)^2
 + 
 \varepsilon^2  \sin^2(\varphi) \left(  \log r \right)^2 
 - \varepsilon  
 (1-\varepsilon) \cos(\varphi)  \log r  \right)  \right],
\end{split}
\]
which is strictly positive for sufficiently large $r$,  but not necessarily diverging to infinity as $r\to\infty$. 
\\
Using a partition of unity 
 $\chi_1, \chi_2\equiv 1-\chi_1,$ with twice differentiable and bounded derivatives of the function
 $\chi_1,$ equalling to one outside a ball of radius $R\in(1,\infty)$ and zero inside a ball of radius $R/2$, we can decompose the interaction as follows
 \[U=U_1+U_2\equiv U\chi_1 + U \chi_2.\]
 With this decomposition, we have
 $U_1,$ an Adams' regular potential, for which we can choose $R\in(0,\infty)$ so that
 \[ \mathcal{V}_1\equiv \frac14|\nabla U_1|^2-\frac14\Delta U_1 \geq m,\]
 for some constant $m\in(0,\infty)$.
 Then, for the measure $d\mu_1\equiv \frac{1}{Z_1}e^{-U_1}d\lambda$ with $Z_1\equiv \int e^{-U_1}d\lambda$, we have the following bound, \cite{HZ}, 
\[\mu_1\left(f^2\mathcal{V}_1\right)\leq \mu_1 |\nabla f|^2.\]
Given the lower bound $\mathcal{V}_1\geq m>0$, this immediately implies Poincar\'e inequality,
\[
m\mu_1\left(f-\mu_1 f\right)^2 \leq \mu_1 |\nabla f|^2.
\]
Since $U_2$ is bounded, we can now apply the bounded perturbation lemma  to arrive at
\end{example}
\begin{thm}
There exists a constant $\tilde m\in(0,\infty)$ such that
\[ \tilde m \mu (f-\mu f)^2 \leq \mu |\nabla f|^2,\]
for any function $f$ for which the right-hand side is well-defined.\\
\end{thm}
\textbf{Nonregular Interactions in Euclidean Space.}
\begin{example} \label{E.g.2}
For $\varepsilon\in(0,1)$ and $\alpha\in(0,\infty)$, let 
\[
U(\mathbf{x})\equiv r^{\alpha}\left(1+\varepsilon\cos (\omega r^{\kappa}) \right)
.
\]
Let
\[
\mathcal{V}\equiv \frac14 |\nabla U|^2 -\frac12 \Delta U.
\]
We have
\[\begin{split}
    \nabla U&=  r^{\alpha-1}\left(\alpha \left(1+\varepsilon\cos (\omega r^{\kappa})\right) -\varepsilon\omega \kappa r^{\kappa}  \sin (\omega r^{\kappa}) \right)\frac{\mathbf{x}}{r} ,
\end{split}
\]
and hence,
 \[\begin{split}
 |\nabla U|^2 &=   r^{2(\alpha-1)}\left(\alpha \left(1+\varepsilon\cos (\omega r^{\kappa}) \right)-\varepsilon\omega \kappa r^{\kappa}  \sin (\omega r^{\kappa}) \right)^2.
 \end{split}\]
For the Laplacian of the interaction, we have
\[\begin{split}
\Delta U &=   
  \nabla\cdot\left(r^{\alpha-1}\left(\alpha \left(1+\varepsilon\cos (\omega r^{\kappa})\right) -\varepsilon\omega \kappa r^{\kappa}  \sin (\omega r^{\kappa}) \right)\frac{\mathbf{x}}{r} \right)\\
  &=   
 r^{\alpha-2}(\alpha-1)\left(\alpha \left(1+\varepsilon\cos (\omega r^{\kappa})\right) -\varepsilon\omega \kappa r^{\kappa}  \sin (\omega r^{\kappa}) \right) \\
 &+ r^{\alpha-1}\left(\alpha \left( -\varepsilon\omega\kappa r^{\kappa-1}\sin (\omega r^{\kappa})\right) -\varepsilon\omega \kappa^2 r^{\kappa-1}  \sin (\omega r^{\kappa})   -\varepsilon\omega^2 \kappa^2 r^{2\kappa-1}  \cos (\omega r^{\kappa})
 \right)  \\
 &+r^{\alpha-1}\left(\alpha \left(1+\varepsilon\cos (\omega r^{\kappa})\right) -\varepsilon\omega \kappa r^{\kappa}  \sin (\omega r^{\kappa}) \right)\frac{n-1}{r}.
\end{split}
 \]   
Suppose that $\nabla U=0$, which is the case when
\[
0<\alpha \left(1+\varepsilon\cos (\omega r^{\kappa})\right) =\varepsilon\omega \kappa r^{\kappa}  \sin (\omega r^{\kappa})  ,
\]
and for large $r,$ that allows for (infinitely many ) solutions close to 
$\omega r^{\kappa}\sim 2n\pi$ with $\cos(\omega r^{\kappa})>0,$ as well as for solutions close to
$\omega r^{\kappa}\sim (2n+1)\pi$ with $\cos(\omega r^{\kappa})<0$.
Then, we have
\[\begin{split}
\Delta U &=     
  r^{\alpha-1}\left(\alpha \left( -\varepsilon\omega\kappa r^{\kappa-1}\sin (\omega r^{\kappa})\right) -\varepsilon\omega \kappa^2 r^{\kappa-1}  \sin (\omega r^{\kappa})   -\varepsilon\omega^2 \kappa^2 r^{2\kappa-1}  \cos (\omega r^{\kappa})
 \right) \\
 &= -r^{\alpha-2}\left( \alpha^2 \left(1+\varepsilon\cos (\omega r^{\kappa})\right)
 +\kappa \alpha\left(1+\varepsilon\cos (\omega r^{\kappa})\right)
+\varepsilon\omega^2 \kappa^2 r^{2\kappa }  \cos (\omega r^{\kappa})
 \right) \\
 &= -r^{\alpha-2}\left( \alpha(\alpha  +\kappa ) +\left(\alpha^2  
 +\kappa \alpha  
+ \omega^2 \kappa^2 r^{2\kappa } \right) \varepsilon \cos (\omega r^{\kappa})
 \right), \\
\end{split}
 \]   
 which can be positive or negative depending on the sign of $\cos(\omega r^\kappa),$ with the modulus growing to infinity.
 \end{example}
 Hence, we get the following result:
 
\begin{thm}
If $\kappa>0$,
 there does not exist 
 a constant  $C\in(0,\infty)$ 
such that the following Adams' regularity condition holds
\[
|\Delta U|\leq C(1+|\nabla U|^2) .
\]
\end{thm}
Then also Poincar\'e inequality does not hold, \cite{HZ}.

\vspace{0.55cm}

\textbf{Nilpotent Lie Group Case}\\
With a homogeneous norm $N$, let 
\[
U(\mathbf{x})\equiv N^{\alpha}\left(1+\varepsilon\cos (\omega N^{\kappa}) \right).
\]
Then
\[\begin{split}
    \nabla U&=  N^{\alpha-1}\left(
    \alpha \left(1+\varepsilon\cos (\omega N^{\kappa})\right) -\varepsilon\omega \kappa N^{\kappa}  \sin (\omega N^{\kappa}) \right)\nabla N, 
\end{split}
\]
and hence
 \[\begin{split}
 |\nabla U|^2 &=   N^{2(\alpha-1)}\left(\alpha \left(1+\varepsilon\cos (\omega N^{\kappa}) \right)-\varepsilon\omega \kappa N^{\kappa}  \sin (\omega N^{\kappa}) \right)^2 |\nabla N|^2.
 \end{split}\]
 As follows from the Euclidean case, there is an infinite unbounded set $g_k\in\mathbb{G},$ where the expression in the bracket vanishes. On top of that,
 for smooth $N$, there are additional points where $\nabla N=0$. \\

For example, in the case of the Heisenberg group, the additional factor we pickup  is $|\nabla N|^2 =\frac{|\mathbf{x}|^2}{N^2}$, compared to the Euclidean case when it is equal to one.\\
 For the sub-Laplacian, we have
 \[\begin{split}
    \Delta U&=  N^{\alpha-2}(\alpha-1)\left( \alpha \left(1+\varepsilon\cos (\omega N^{\kappa})\right)  -\varepsilon\omega \kappa N^{\kappa}  \sin (\omega N^{\kappa})  
    \right) |\nabla N|^2 \\
    &+ N^{\alpha-1} \left(   {\alpha} \left(-\varepsilon\omega\kappa N^{\kappa-1}\sin (\omega N^{\kappa})  -\varepsilon\omega  \kappa^2 N^{\kappa-1}  \sin (\omega N^{\kappa}) -\varepsilon\omega^2 \kappa^2 N^{2\kappa-1}  \cos(\omega N^{\kappa})  \right)
    \right)|\nabla N|^2 
    \\
    &+N^{\alpha-1}\left(\alpha \left(1+\varepsilon\cos (\omega N^{\kappa}) \right)-\varepsilon\omega \kappa N^{\kappa}  \sin (\omega N^{\kappa}) \right)\Delta N. 
\end{split}
\]
At the points $g_k\in\mathbb{G}$ where $\nabla U(g_k)=0$, but possibly $\nabla N\neq 0$, we have
\[\begin{split}
    \Delta U(g_k)&=   
    -\varepsilon\kappa\omega N^{\alpha+\kappa-2} \left(   (\alpha+    \kappa )   \sin (\omega N^{\kappa})    + \omega  \kappa  N^{\kappa}  \cos(\omega N^{\kappa})  \right)
    |\nabla N|^2 ,
\end{split}
\]
with two subsequences diverging to $\pm\infty$, respectively.\\
In case of the Folland-Kor\'anyi (Kaplan) norm $N$, we have 
\[
\Delta N = \frac{(Q - 1)}{N}|\nabla N |^2.
\]
In this case, we have
\[\begin{split}
    \Delta U&=   N^{\alpha-2}|\nabla N|^2 (\alpha-1)\left(\alpha \left(1+\varepsilon\cos (\omega N^{\kappa})\right) -\varepsilon\omega \kappa N^{\kappa}  \sin (\omega N^{\kappa})  
    \right) \\
    &+   N^{\alpha-2}|\nabla N|^2 \left(\alpha  \left(-\varepsilon\omega\kappa N^{\kappa }\sin (\omega N^{\kappa})\right) -\varepsilon\omega  \kappa^2 N^{\kappa }  \sin (\omega N^{\kappa}) -\varepsilon\omega^2 \kappa^2 N^{2\kappa }  \cos(\omega N^{\kappa})   
    \right)
    \\
    &+N^{\alpha-2}|\nabla N |^2\left(\alpha \left(1+\varepsilon\cos (\omega N^{\kappa})\right) -\varepsilon\omega \kappa N^{\kappa}  \sin (\omega N^{\kappa}) \right) (Q - 1) ,\\
\end{split}
\]
i.e.
\[\begin{split}
    \Delta U&=   
   N^{\alpha-2}|\nabla N |^2(\alpha^2 (1+\varepsilon\cos (\omega N^{\kappa}))  -N^{\alpha-2}|\nabla N |^2\varepsilon\omega^2 \kappa^2 N^{2\kappa }  \cos(\omega N^{\kappa})\\
    &-  N^{\alpha-2}|\nabla N |^2\varepsilon\omega \kappa (\alpha+\kappa+Q)N^{\kappa}  \sin (\omega N^{\kappa})).
\end{split}
\]
Hence,
\[\begin{split}
 \frac14|\nabla U|^2 -\frac12\Delta U &=   \frac14N^{2(\alpha-1)} |\nabla N|^2 
 \left(\alpha \left(1+\varepsilon\cos (\omega N^{\kappa}) \right)-\varepsilon\omega \kappa N^{\kappa}  \sin (\omega N^{\kappa}) \right)^2 \\
 &-\frac12 N^{\alpha-2}|\nabla N |^2 \left(\alpha^2 \left(1+\varepsilon\cos (\omega N^{\kappa})\right)  %
 - \varepsilon\omega^2 \kappa^2 N^{2\kappa }  \cos(\omega N^{\kappa})\right) \\
    &-  \frac12 N^{\alpha-2}|\nabla N |^2\varepsilon\omega \kappa (\alpha+\kappa+Q)N^{\kappa}  \sin (\omega N^{\kappa}) . 
 \end{split}\]
 This also vanishes everywhere where $\nabla N=0$.
 As we showed in Section \ref{sec6} in the case of nilpotent Lie groups of second order, Adams' regularity condition fails because mixed derivatives of $N$ of second order cannot be globally bounded by $|\nabla N|^{2}$.\\
 Finally, we remark that even bounded perturbations of polynomial interaction may break down the Adams' regularity condition. We note, however, that bounded perturbations of Adams' regular interaction can be handled separately. For example,
 consider
 \[
 U(d)\equiv d^p + \cos(\omega d^\kappa),
 \]
 with some good distance function $d$ for which gradient does not vanish and the second derivatives behave well outside some ball.
 Then, for $\kappa = p$, the sets of zeros of $\nabla U$ is unbounded, infinite, and there exists two sequences  of points for which the corresponding values of $\Delta U$ diverge to $\pm\infty$, respectively. \\

\begin{example} \label{E.g.3}

Let $d\mu\equiv \frac1{Z}e^{-\mathcal{U}}d\lambda$ where 
\[\mathcal{U}=\frac1{2n}\left(x^2+y^2+2z^2\right)^n\equiv \frac1{2n} r^{2n} \]
and $Z$ is a normalisation constant.
Then we have
\[\begin{split}
    X\mathcal{U}&= r^{2(n-1)}(x+yz),\\    Y\mathcal{U}&= r^{2(n-1)}(y-xz)
\end{split}\]

\[|\nabla   \mathcal{U}|^2 
= r^{4(n-1)}\left(|x+yz|^2 +|y-xz|^2\right)=
 r^{4(n-1)}|\mathbf{x}|^2 \left(1+z^2\right)
 \]
and
\[\begin{split}
    X^2\mathcal{U}&= X(r^{2(n-1)}(x+yz))=
    (n-1)r^{2(n-2)}2(x+yz)^2+ r^{2(n-1)}\left(1+\frac12 y^2\right),
    \\   
     YX \mathcal{U}&=Y(r^{2(n-1)}(x+yz))=
    (n-1)r^{2(n-2)}2(x+yz)(y-xz)+ r^{2(n-1)}\left(z-\frac12 xy\right),
    \\
    Y^2\mathcal{U}&= Y(r^{2(n-1)}(y-xz))=
   (n-1) r^{2(n-2)}2(y-xz)^2 +r^{2(n-1)}\left(1+\frac12 x^2\right)
   \\
    XY \mathcal{U}&= X(r^{2(n-1)}(y-xz))=
   (n-1) r^{2(n-2)}2(x+yz)(y-xz)  +r^{2(n-1)}\left( -z-\frac12xy\right)
\end{split}\]
Hence
\[\Delta   \mathcal{U}= 2(n-1) r^{2(n-2)}  |\mathbf{x}|^2 (1 +z^2)   +r^{2(n-1)}\left(2+\frac12  |\mathbf{x}|^2\right), \]
and
\[\begin{split}
|YX \mathcal{U}|^2+ |XY \mathcal{U}|^2&=
8(n-1)^2 r^{4(n-2)} (x+yz)^2(y-xz)^2  \\
&-4(n-1) r^{2(2n-3)} (x+yz)(y-xz)  xy 
\\
&+  r^{4(n-1)}\left(z-\frac12xy\right)^2 +  r^{4(n-1)}\left(z+\frac12xy\right)^2
\end{split}\]
In general, the last formula cannot be dominated by the square of the gradient for small 
$|\mathbf{x}|^2,$ and hence, Adams' regularity condition does not hold.\\
Let
\[ \begin{split}\mathcal{V}&\equiv
\frac14 |\nabla   \mathcal{U}|^2 - \frac12 \Delta   \mathcal{U} \\
&= \frac14  r^{4(n-1)}|\mathbf{x}|^2 \left(1+z^2\right) 
-\frac12 \left(2(n-1) r^{2(n-2)}  |\mathbf{x}|^2 (1 +z^2)   +r^{2(n-1)}\left(2+\frac12  |\mathbf{x}|^2\right)\right),
\end{split} \]
and we have
\[ \tag{I}
\mu \mathcal{V}f^2\leq \mu|\nabla  f|^2.
\]

On the other hand, for Hardy's inequality, we have
\[
\int \frac{1}{|\mathbf{x}|^2} f^2e^{-\mathcal{U}} d\lambda   \leq \frac{C}{(\tilde{n}-2)^2}  \int |\nabla   f -f\nabla   \mathcal{U}|^2 e^{-\mathcal{U}} d\lambda   ,
\]
and hence
\[
\mu \frac{1}{|\mathbf{x}|^2} f^2\leq 
\frac{2C}{(\tilde{n}-2)^2}\mu|\nabla   f|^2 + \frac{2C}{(\tilde{n}-2)^2}\mu\left(f^2|\nabla   \mathcal{U}|^2\right).\]
Thus, in our example, we obtain
\[\tag{II} 
\mu \frac{1}{|\mathbf{x}|^2} f^2\leq 
\frac{2C}{(\tilde{n}-2)^2}\mu|\nabla   f|^2 + 
\frac{2C}{(\tilde{n}-2)^2}\mu\left(f^2 r^{4(n-1)}|\mathbf{x}|^2 \left(1+z^2\right)\right).
\]
Combining (I) and (II), with some $A^2   \in(0,\infty)$ to be chosen later, we obtain
\[ \tag{III} \begin{split}
\mu f^2\left( \mathcal{V} f{x}^2  
+\frac{A^2}{|\mathbf{x}|^2} \right) &\leq 
\frac{2CA^2}{(\tilde{n}-2)^2}\mu|\nabla   f|^2 \\
&\qquad + 
\frac{2CA^2}{(\tilde{n}-2)^2}\mu\left(f^2 r^{4(n-1)}|\mathbf{x}|^2(1+z^2)\right).
\end{split}\]
After rearrangement, this yields
\[\tag{IV} \begin{split}
&
\mu f^2\left( 
\left(\left(\frac14  -\frac{2CA^2}{(\tilde{n}-2)^2}\right)
r^{4(n-1)}  \left(1+z^2\right) 
-  \left( (n-1) r^{2(n-2)}   (1 +z^2)   +\frac1{4}r^{2(n-1)}  \right)\right) |\mathbf{x}|^2
+\frac{A^2}{|\mathbf{x}|^2} 
\right)  
-\mu f^2 r^{2(n-1)}
\\
&\leq 
\frac{2CA^2}{(\tilde{n}-2)^2}\mu|\nabla f|^2 .
\end{split}
\]
If $n$ is large enough, so that
\[
\frac14  - \frac{2CA^2}{(\tilde{n}-2)^2} >\frac1{8},
\]
then  there exists a constant $B\in(0,\infty)$  such that we have
\[
\mathcal{W}\equiv  
\left(\left(\frac14  -\frac{2CA^2}{(\tilde{n}-2)^2}\right)
r^{4(n-1)}  \left(1+z^2\right) 
-  \left( (n-1) r^{2(n-2)}   (1 +z^2)   +\frac1{4}r^{2(n-1)}  \right)\right) + B > \frac1{16} r^{4(n-1)}  \left(1+z^2\right)  ,
\]
and hence
\[\tag{V}  
\mu f^2\left( 
\mathcal{W}  |\mathbf{x}|^2
+\frac{A^2}{|\mathbf{x}|^2} 
- (1+B)\mu f^2 r^{2(n-1)} \right)
\leq 
\frac{2CA^2}{(\tilde{n}-2)^2}\mu|\nabla  f|^2 
.
\]
At this point, we note that the following inequality holds
\[
\mathcal{W}|\mathbf{x}|^2 + 
\frac{A^2}{|\mathbf{x}|^2} \geq 
\frac1{16} r^{4(n-1)} |\mathbf{x}|^2(1+z^2) + \frac{A  }{|\mathbf{x}|^2} \geq
\frac{A}2 r^{2(n-1)} (1+z^2)^{\frac12}.
\]
This, together with (IV),  yields
\[\tag{$\ast$}
 \mu\left( f^2 \left(\frac{A}{2}r^{2(n-1)} (1+z^2)^{\frac12} - (1+B) r^{2(n-1)}\right)\right) \leq  %
\frac{4A^2}{A-1} \frac{C}{(\tilde{n}-2)^2} 
\mu|\nabla f|^2 .
\]
Since there exists a constant $D\in(0,\infty)$ such that 
\[\frac{A}{4}r^{2(n-1)} (1+z^2)^{\frac12} - (1+B) r^{2(n-1)} +D\geq 0\]
the following result is true.
\end{example}
\begin{thm} 
Let $d\mu\equiv \frac1{Z}e^{-\mathcal{U}}d\lambda$ where 
\[\mathcal{U}=\frac1{2n}\left(x^2+y^2+2z^2\right)^n\equiv \frac1{2n} r^{2n} \]
and $Z$ is a normalisation constant.\\
Suppose Hardy's inequality for the Haar measure and subgradient holds with the constant $\frac{C}{(n-2)^2}$ and for a given constant $A\in(0,\infty)$
\[
\frac14  - \frac{2CA^2}{(\tilde{n}-2)^2} >\frac1{8},
\]
then there exists a constant $D\in(0,\infty)$ such that
\[\tag{$\star$}
 \mu\left( f^2 \left(\frac{A}{4}r^{2(n-1)} (1+z^2)^{\frac12}  \right)\right) \leq  %
\frac{4A^2}{A-1} \frac{C}{(\tilde{n}-2)^2} 
\mu|\nabla f|^2 + D \mu  f^2.
\]
\end{thm}
\vspace{0.5cm}

\section{Harmonic Polynomials for Nilpotent Lie Algebras}

Consider the Heisenberg group $\mathbb{H}_{1},$ with the generators
$X=\partial_{x}+\frac{1}{2}y\partial_{z},$ and $Y=\partial_{y}-\frac{1}{2}x\partial_{z}.$
Let $\triangle_{\mathcal{L}}\equiv X^2+Y^2$ be the corresponding sub-Laplacian.
$W_{\pm}\equiv4z\pm2xy$ are said to be harmonic polynomials since
we have:

\[\begin{split}
X(4z-2xy)&=0,\:\:\:\:\:\:\:\:\:\:\:\:\:\:\:\;\:\:Y(4z+2xy)=0,
\\
X(4z+2xy)&=4y,\:\:\:\:\:\:\:\:\:\:\:\:\:\:\:\;\:\:Y(4z-2xy)=-4x 
\end{split}\]
and hence,
\[
X^{2}W_{\pm}=0,\:\:\:\:\:\:\:\:\:\:\:\:\:\:\:\;\:\:Y^{2}W_{\pm}=0,
\]

so $\Delta_{\mathcal{L}}W_{\pm}=0.$
\\
\begin{rem} We remark that we have encountered the functions $W_\pm$ before. Namely, in Example \ref{E.g.3}  with $n=1$ we have
\[\mathcal{U}=\frac1{2}\left(x^2+y^2+2z^2\right)   \]
Hence
\[X\mathcal{U} =x+yz, \qquad Y\mathcal{U} =y-xz\]
and hence
\[YX\mathcal{U} =Y(x+yz)=z-\frac12xy, \qquad XY\mathcal{U} =X(y-xz)=-(z+\frac12xy).\]
\end{rem}
Another harmonic polynomial is the following: $V\equiv x^{2}-y^{2}.$ In fact, $X^{2}V=2,$ and $Y^{2}V=-2,$ so $\Delta_{\mathbb{\mathcal{L}}}V=(X^{2}+Y^{2})V=0.$ Using these polynomials, we define the following homogeneous norm $N^4\equiv W_{+}^{2}+W_{-}^{2}+2V^{2}.$
Computing, we observe the following:

\noindent 
\[\begin{split}
N^4&=W_{+}^{2}+W_{-}^{2}+2V^{2}
\\
&=(4z+2xy)^{2}+(4z-2xy)^{2}+2(x^{2}-y^{2})^{2}
\\
&=32z^{2}+8x^{2}y^{2}+2x^{4}+2y^{4}-4x^{2}y^{2}
\\
&=2(16z^{2}+x^{4}+y^{4}-2x^{2}y^{2})
\\
&=2((x^{2}+y^{2})^{2}+16z^{2})
\\
&=2N_{kaplan}^{4}.
\end{split}\]

Now, consider the Heisenberg group $\mathbb{H}_{n},$ with the generators
$X_{i}=\partial_{x_{i}}+\frac{(-1)^{i}}{2}x_{2n-i+1}\partial_{z},$
for $i=1,...,2n.$

$W_{\pm}\equiv4z\pm2\sum_{j=1}^{n}x_{j}x_{2n-j+1}$ are said to be
harmonic polynomials since we have:

\[
X_{i}W_{+}=2(-1)^{i}x_{2n-i+1}+2x_{2n-i+1},\,\,\,\,\,\,\,\,\,\,\,\,\,\,\,\,\,\,\,\,\,\,\,\,\,\,\,\,\,\,\,\,X_{i}W_{-}=2(-1)^{i}x_{2n-i+1}-2x_{2n-i+1}
\]

Hence,
\[
X_{i}^{2}W_{\pm}=0
\]

so $\Delta_{\mathbb{\mathcal{L}}}W_{\pm}=0.$

Other harmonic polynomials are the following: $V_{j,k}\equiv(x_{j}-x_{2n-j+1})(x_{k}+x_{2n-k+1}).$

If $j\neq k,$ then, $X_{j}^{2}V_{j,k}=0.$ 

If $j=k,$ then $X_{j}^{2}V_{j}=2,$ and $X_{2n-j+1}^{2}V_{j}=-2.$
Hence, $\Delta_{\mathbb{\mathcal{L}}}V_{j,k}=0.$

Let $V\equiv\sum_{j=1}^{n}\sum_{k=1}^{n}V_{j,k}^{2}.$

Using these polynomials, we define the following homogeneous norm $N=W_{+}^{2}+W_{-}^{2}+2V^{2}.$
Computing, we observe the following:

\begin{singlespace}
\noindent 
\[\begin{split}
N^4&=W_{+}^{2}+W_{-}^{2}+2V^{2}
\\
&=32z^{2}+8(\sum_{j=1}^{n}x_{j}x_{2n-j+1})^{2}+2\sum_{j=1}^{n}\sum_{k=1}^{n}\left[(x_{j}-x_{2n-j+1})(x_{k}+x_{2n-k+1})\right]^{2}
\\
&=2\left(16z^{2}+4(\sum_{j=1}^{n}x_{j}x_{2n-j+1})^{2}+\left(\sum_{j=1}^{n}(x_{j}-x_{2n-j+1})^{2}\right)\left(\sum_{j=1}^{n}(x_{j}+x_{2n-j+1})^{2}\right)\right)
\\
&=2\left(16z^{2}+4(\sum_{j=1}^{n}x_{j}x_{2n-j+1})^{2}+\left(\sum_{j=1}^{n}x_{j}^{2}+x_{2n-j+1}^{2}-2x_{j}x_{2n-j+1}\right)\left(\sum_{j=1}^{n}x_{j}^{2}+x_{2n-j+1}^{2}+2x_{j}x_{2n-j+1}\right)\right)
\\
&=2\left(16z^{2}+4(\sum_{j=1}^{n}x_{j}x_{2n-j+1})^{2}+\left(\sum_{j=1}^{2n}x_{j}^{2}-2\sum_{j=1}^{n}x_{j}x_{2n-j+1}\right)\left(\sum_{j=1}^{2n}x_{j}^{2}+2\sum_{j=1}^{n}x_{j}x_{2n-j+1}\right)\right)
\\
&=2\left(16z^{2}+\left(\sum_{j=1}^{2n}x_{j}^{2}\right)^{2}\right)
\\
&=2N_{kaplan}^{4}
\end{split}\]

\end{singlespace}

\begin{thm}
In the Heisenberg group $\mathbb{H}^n,$ the Kaplan norm  $N_{Kaplan}$ satisfies
 \[N^4_{Kaplan} =\frac12 \left( W_{+}^{2}+W_{-}^{2}+2V^{2}  \right). \] 
\end{thm} 

\vspace{2cm}
\noindent{\bf Acknowledgements:} E.B.D. is funded by the European Union's Horizon 2020 research and innovation programme under the Marie Sk\l odowska-Curie grant agreement n\textsuperscript{o}101034255.\\[-2pt]

\vspace{-0.5500cm}

\end{document}